\documentclass[twoside,leqno,symbols-for-thanks]{rmi-adapt}
\usepackage[colorlinks,linkcolor=black,citecolor=black,urlcolor=black, linktocpage=true,dvips]{hyperref}
\usepackage{srcltx}
\usepackage[english]{babel}
\numberwithin{equation}{section}
\setinitialpage{1}
\renewcommand{\qedsymbol}{$\Box$}

\title[Brownian motion on treebolic space]
{Brownian motion on treebolic space:\\
escape to infinity}

\author[A. Bendikov, L. Saloff-Coste, M. Salvatori and W. Woess]
{Alexander Bendikov, Laurent Saloff-Coste,\\ Maura Salvatori and Wolfgang Woess}

\address[bendikov@math.uni.wroc.pl]{{\sc Alexander Bendikov}: 
Institute of Mathematics, Wroclaw University, \\
Pl. Grundwaldzki 2/4, 50-384 Wroclaw, Poland.}

\address[lsc@paris.math.cornell.edu]{{\sc Laurent Saloff-Coste}: 
Department of Mathematics, Cornell University, \\ 
Ithaca, NY 14853, USA.}

\address[maura.salvatori@unimi.it]{{\sc Maura Salvatori}: 
Dipartimento di Matematica, Universit\`a di Milano, \\
Via Saldini 50, 20133 Milano, Italy.}

\address[woess@TUGraz.at]{{\sc Wolfgang Woess}: 
Institut f\"ur Mathematische Strukturtheorie, \\
Technische Universit\"at Graz,
Steyrergasse 30, A-8010 Graz, Austria.}

\amsclassification[60J65, 53C23, 20F65, 05C05]{60J25}


\usepackage{pictex}  

\newcommand\ab{\mathfrak{a}}
\newcommand\Af{\mathcal{A}}
\newcommand\Aff{\operatorname{\sf Aff}}
\newcommand\al{\alpha}
\newcommand\AND{\quad\text{and}\quad}
\newcommand\assim{\stackrel{\text{\rm a.s.}}{\sim}}
\newcommand\asymplaw{\stackrel{\text{\rm in law}}{\asymp}}
\newcommand\Aut{\operatorname{\sf Aut}}
\newcommand\bb{\mathfrak{b}}

\newcommand\bd{\partial}
\newcommand\binfty{\boldsymbol{\infty}}
\newcommand\BS{\operatorname{\sf BS}}

\newcommand\CC{\mathcal C}
\newcommand\cf{\curlywedge}

\newcommand\DC{\mathcal D}
\newcommand\de{\delta}
\newcommand\DL{\mathsf{DL}}
\newcommand\dist{\mathsf{d}}
\newcommand\Dom{\text{\rm Dom}}

\newcommand\ep{\varepsilon}
\newcommand\eqlaw{\stackrel{\text{\rm in law}}{=}}
\newcommand\Ex{\mathsf{E}}
\newcommand\fb{\mathbf f}
\newcommand\ga{\gamma}
\newcommand\geo[1]{\overline{\rl #1}}

\newcommand\gf{\mathfrak{g}}
\newcommand\Hb{{\mathbb H}}

\newcommand\hb{\mathbf h}
\newcommand\hor{\mathfrak{h}}

\newcommand\HT{\operatorname{\sf HT}}

\newcommand\im{\mathfrak{i}\,}
\newcommand\IM{\operatorname{\text{\sl Im}}}
\newcommand\Lap{\Delta}
\newcommand\Lf{\mathsf{L}}
\newcommand\la{\lambda}
\newcommand\LC{\mathcal L}

\newcommand\LH{\operatorname{\sf L}\!\Hb}
\newcommand\LT{\operatorname{\sf LT}}

\newcommand\md{\boldsymbol\delta}
\newcommand\mmu{\boldsymbol{\mu}}
\newcommand\ms{\mathbf m}
\newcommand\N{\mathbb N}
\newcommand\nb{\mathbf{n}}
\newcommand\nn{\mathsf{n}}

\newcommand\of{\mathfrak{o}}
\newcommand\ol{\overline}
\newcommand\om{\varpi}
\newcommand\Prob{\mathsf{Pr}}
\newcommand\pb{\mathbf p}
\newcommand\pp{\mathsf{p}}
\newcommand\qq{\mathsf{q}}
\newcommand\Q{\mathbb{Q}}
\newcommand\R{\mathbb{R}}
\newcommand\RE{\operatorname{\text{\sl Re}}}
\newcommand\rha{\mathsf{a}}

\newcommand\rl{\rule[0pt]{0pt}{7pt}}
\newcommand\RR{\boldsymbol{R}}
\newcommand\scs{\scriptstyle}
\newcommand\Sf{\mathsf{S}}
\newcommand\simlaw{\stackrel{\text{\rm in law}}{\sim}}
\newcommand\Sol{\text{\sf Sol}}
\newcommand\Stab{\text{\sf Stab}}
\newcommand\T{\mathbb T}
\newcommand\tolaw{\stackrel{\text{\rm in law}}{\to}}
\newcommand\ttau{\boldsymbol{\tau}}
\newcommand\uf{\mathfrak{u}}
\newcommand\ul{\underline}
\newcommand\uno{\mathbf{1}}
\newcommand\Var{\mathsf{Var}}
\newcommand\vf{\mathfrak{v}}

\newcommand\wf{\mathfrak{w}}
\newcommand\wh{\widehat}
\newcommand\wt{\widetilde}

\newcommand\XX{\boldsymbol{X}}

\newcommand\zf{\mathfrak{z}}
\newcommand\Z{\mathbb Z}

\newtheoremstyle{mythm}
  {9pt}
  {9pt}
  {\itshape}
  {0pt}
  {\bfseries}
  {}
  { }
  {\thmnumber{(#2)}\thmname{ #1}\thmnote{ #3}}

\newtheoremstyle{mydef}
  {9pt}
  {9pt}
  {\normalfont}
  {0pt}
  {\bfseries}
  {}
  { }
  {\thmnumber{(#2)}\thmname{ #1}\thmnote{ #3}}

\theoremstyle{mythm}
\newtheorem{thm}[equation]{Theorem.}
\newtheorem{pro}[equation]{Proposition.}
\newtheorem{lem}[equation]{Lemma.}
\newtheorem{cor}[equation]{Corollary.}

\theoremstyle{mydef}
\newtheorem{dfn}[equation]{Definition.}

\newtheorem{imp}[equation]{}
\newtheorem{rmk}[equation]{Remark.}

\begin{document}

%
%

\begin{abstract}
Treebolic space is an analog of the {\sf Sol} geometry, namely, it is the
horocylic product of the hyperbolic upper half plane $\Hb$ and the homogeneous tree 
$\T=\T_{\pp}$ with  degree $\pp+1 \ge 3$, the latter seen as a one-complex. Let
$\hor$ be the Busemann function of $\T$ with respect to a fixed boundary point.
Then for real $\qq > 1$ and integer $\pp \ge 2$,  treebolic space $\HT(\qq,\pp)$ 
consists of all pairs $(z=x+\im y,w) \in \Hb \times \T$ with $\hor(w) =  \log_{\qq} y$.
It can also be obtained by glueing together horizontal strips of $\Hb$ in a 
tree-like fashion. We explain the geometry and metric of $\HT$ and exhibit a 
locally compact group of isometries (a horocyclic product of affine groups) that 
acts with compact quotient. When $\qq=\pp$, that group contains the amenable 
Baumslag-Solitar group $\BS(\pp)$ as a co-compact lattice, while when 
$\qq \ne \pp$, it is amenable, but non-unimodular. $\HT(\qq,\pp)$ is a key example
of a strip complex in the sense of \cite{BSSW}.

 Relying on the analysis of strip complexes developed by the same authors 
in \cite{BSSW}, we consider a family of natural Laplacians with ``vertical drift'' 
and describe the associated Brownian motion. The main difficulties come from the singularites
which treebolic space (as any strip complex) has along its bifurcation lines. 

In this first part, we obtain the rate of escape and
a central limit theorem, and describe how Brownian motion converges to the
natural geometric boundary at infinity. Forthcoming work will be dedicated to positive 
harmonic functions. 
\end{abstract}

%
%

\section{Introduction}\label{sec:intro}
Let $\Hb = \{ x + \im y : x \in \R\,,\; y > 0 \}$ be hyperbolic upper half space, and
$\T=\T_{\pp}$ be the homogeneous tree, drawn in such a way that every vertex of
$\T$ has one predecessor and $\pp$ successors.  \emph{Treebolic space} is a Riemannian
2-complex, a \emph{horocyclic product} of  $\Hb$ and $\T$. Let us start with a picture
and an informal description.

Let $1 < \qq\in \R$. Subdivide $\Hb$ into the strips 
$\Sf_k = \{ x+\im y : x \in \R\,,\; \qq^{k-1} \le y \le \qq^{k} \}$, 
where $k \in \Z$. (See Figure~3 further below.)
Each strip is bounded by two horizontal lines of the form
$\Lf_k= \{ x+\im \qq^k : x \in \R \}$, which in hyperbolic geometry are
horocycles with respect to the ``upper'' boundary point  $\binfty$ 
(or rather $\im\infty$).
In treebolic space $\HT(\qq,\pp)$, infinitely many copies of those strips
are glued together in a tree-like fashion: for each $k \in \Z$, the bottom lines of 
$\pp$ copies of $\Sf_{k}$ are identified with each other and with the top line
of a copy of $\Sf_{k-1}$. Thus, every copy of any of the $\Lf_k$ becomes a
\emph{bifurcation line} whose ``side view'' is a vertex $v$ of $\T$ that
can be used to identify the line as $\Lf_v$ (instead of $\Lf_k$). In the same
way, we write $\Sf_v$ for the strip sitting below $\Lf_v$ in our picture.
Each strip is equipped with the standard hyperbolic length element, 
and combining this with the tree metric, one obtains a natural metric on
$\HT(\qq,\pp)$. A more formal description will be given in 
\S \ref{sec:geometry}.

$$
\beginpicture 

\setcoordinatesystem units <0.95mm,0.95mm>

\setplotarea x from -16 to 54, y from 0 to 48

\plot 0 0  40 20  48 36  8 16  0 0  -8 16  -2 28  38 48  36.666 45.333 /

\plot 48 36  54 48  14 28  8 16  2  28  42 48   44.4 43.2 /

\plot 4.8 22.4  -8 16  -14 28  26 48  28.4  43.2 /

\put{$\leftarrow$ copies of $\Sf_{k-1}$} [l] at 46 28
\put{$\leftarrow$ copies of $\Lf_k$} [l] at 50 36
\put{$\leftarrow$ copies of $\Sf_{k}$} [l] at 53 42

\endpicture
$$
\begin{center}
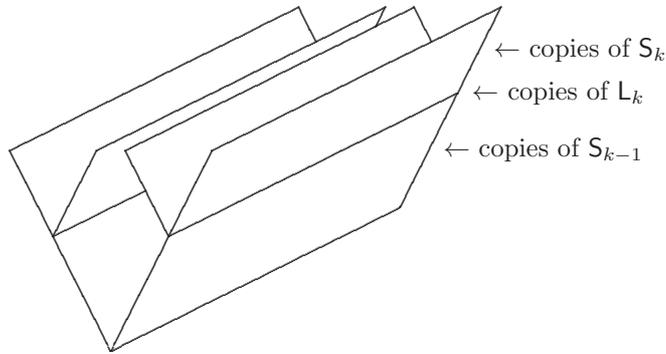
{\sl Figure~1.\/} A finite section of treebolic space, with $\pp=2$.
\footnote{This figure also appears in \cite{BSSW}.}
\end{center}

\medskip


Why is this space interesting$\,$? First of all, it is a key example of a 
\emph{strip complex} in the sense of \cite{BSSW}. Strip complexes are a class
of Riemannian complexes. Laplacians and the associated potential theory 
on Riemannian complexes appear in the book of {\sc Eells and Fuglede}~\cite{EeFu}.
A study of Brownian motion and harmonic functions on Euclidean complexes was 
undertaken by {\sc Brin and Kifer}~\cite{BrKi}. In \cite{BSSW}, the theory of
Laplacians and diffusion on strip complexes, properties of the heat kernel,
etc., were studied in a careful and rigorous way. In this spirit, the present
paper is the first detailed case study of what can be achieved on the basis of  
that theory, which provides a highly non-trivial extension of the very popular subject of
analysis and probability on ``quantum graphs'' (metric graphs) to what one might also 
call ``quantum complexes''. 

Second, treebolic space is a horosphere in the product
space $\Hb \times \T$, where the tree $\T$ is viewed as a one-dimensional complex
in which each edge is a copy of a suitable compact interval.
In other words, it is the \emph{horocyclic product} of $\Hb$ and $\T$. A first
appearance of such a  horocyclic product was that of two trees with 
(integer) branching numbers $\pp$ and $\qq \ge 2$, respectively. This is the
Diestel-Leader graph $\DL(\pp,\qq)$, which for $\pp \ne \qq$ was proposed 
by {\sc Diestel and Leader}~\cite{DiLe} as a candidate example to answer
the following question of {\sc Woess}~\cite{SoWo}: is there a vertex-transitive
graph which is not quasi-isometric with a Cayley graph? It was only quite
recently that {\sc Eskin, Fisher and Whyte} 
finally showed, as
part of impressive work ~\cite{EFW1}, \cite{EFW2}, \cite{EFW3}, 
that $\DL(\pp,\qq)$ is indeed such an example.
On the other hand, in case of equal branching numbers of the
two trees, $\DL(\pp,\pp)$ is a Cayley graph of the lamplighter group
$\Z(\pp) \wr \Z$. This geometric realisation of the latter lead to
a good understanding of random walks, spectra and boundary theory of
those groups, the DL-graphs, and of horocyclic products of more than 2 trees,
see the work of {\sc Bertacchi, Bartholdi, Brofferio, Neuhauser, Woess}
\cite{Ber}, \cite{Wlamp}, \cite{BaWo}, \cite{BrWo1}, \cite{BrWo2}, \cite{BNW}.

Besides $\DL(\pp,\qq)$, treebolic space has another, more classical
sister structure. This is $\Sol(\pp,\qq)$, the horocyclic product of 
two hyperbolic planes
with curvatures $-\pp^2$ and $-\qq^2$, respectively, where $\pp, \qq > 0$. 
Besides being a $3$-dimensional Riemannian manifold, $\Sol(\pp,\qq)$ 
can be seen
as a Lie group, which is the semidirect product of $\R$ with $\R^2$
induced by the action $(x,y) \mapsto (e^{\pp z}x,e^{-\qq z}y)$, $z \in \R$.
$\Sol(1,1)$ is one of Thurston's eight model geometries in dimension 3.
The Brownian motion generated by the Laplace-Beltrami operator on $\Sol(\pp,\qq)$ 
is studied in detail in a sister
paper to the present one, by {\sc Brofferio, Salvatori and Woess}~\cite{BSW}.

The analogy between $\DL(\pp,\qq)$ and $\Sol(\pp,\qq)$ becomes also apparent
in \cite{EFW1},  \cite{EFW2}, \cite{EFW3}, where the quasi-isometry classes of these
graphs, resp. manifolds are determined. Coming back to treebolic space, we mention
that like the above sister structures, it is neither Gromov hyperbolic nor Cat(0).
We shall explain below that the amenable Baumslag-Solitar group
$\BS(\pp) = \langle a,b \mid ab = b^{\pp}a \rangle$ acts on $\HT(\pp,\pp)$
by isometries and with compact quotient. This fact has been exploited by
{\sc Farb and Mosher~\cite{FaMo}} (without describing the space as
a horocyclic product) 
in order to determine the quasi-isometry types 
of the Baumslag-Solitar groups. On the other hand, we shall see that
for $\pp \ne \qq$, no discrete group can act in such a way on our space. 

\smallskip

In the present paper, in \S \ref{sec:geometry} we first exhibit more details 
about the geometry of treebolic
space and its metric $\dist(\cdot,\cdot)=\dist_{\HT}(\cdot,\cdot)$ and explain its 
isometry group, which is (up to the obvious reflections
with respect to vertical hyperplanes) obtained as a ``horocyclic'' product
of the group $\Aff(\Hb,\qq)$ of all affine mappings $z \mapsto \qq^k z + b$ 
($z \in \Hb$, $k \in \Z$, $b\in \R$) and the affine group of the tree
$\T_{\pp}$, that is, the group of all automorphisms (self-isometries) 
of the tree that fix a given boundary point. 

We next, in \S \ref{sec:Laplacians}, turn our attention to the Laplace operator 
on $\HT$, whose rigorous construction as an essentially self-adjoint
diffusion operator bears a serious challenge in view of the singularities
which our structure has along the bifurcation lines. This challenge was faced in the 
general setting of strip complexes in \cite{BSSW}. As a matter of fact, we consider a 
family of Laplacians $\Lap_{\alpha,\beta}$ 
with two ``vertical drift'' parameters $\alpha$ and $\beta$.
When looking at Bownian motion (BM), that is, the diffusion 
$(X_t)_{t \ge 0}$ on $\HT$ whose infinitesimal
generator is $\Lap_{\alpha,\beta}\,$, it is hyperbolic BM 
with linear drift parameter $\alpha$
in the interior of each strip. On the other hand, 
$\beta$ is responsible for the random choice of the 
strip into which BM should make its next infinitesimal step when the current
position is on one of the bifurcation lines. The overall drift relies on both
in terms of the number $\rha = \beta \pp \,\qq^{\alpha-1}$.
The drift is $0$ if and only if $\rha =1$, while BM has an ``upwards'' (resp. ``downwards'')
drift when $\rha > 1$ (resp. $<1$).
The Laplacian and Brownian motion admit natural projections on $\T$ and $\Hb$,
as well as on $\R$. The projection onto $\R$ associates with each point its
height: it is the Busemann function with respect to the boundary
point at infininty of $\Hb$ (as well as of the tree). The projected
Brownian motion $(Z_t)_{t \ge 0}$ on $\Hb$ is in general not ordinary hyberbolic BM with drift
parameter $\alpha$, except when $\beta = 1/\pp$. 
That is, it evolves like hyperbolic BM with drift in the interior of each
of the strips $\Sf_k$ into which $\Hb$ has been ``sliced'', while it receives
an additional vertical ``kick'' (absent only when $\beta=1/\pp$) on each of the lines $L_k$.

The projection $(W_t)_{t \ge 0}$ on $\T$ is a typical example of BM on a
metric graph (the tree). The study of the corresponding Laplace operators
is by now well established, and more straightforward than the higher dimensional
version on strip complexes  that we are dealing with here.
See e.g. {\sc Cattaneo}~\cite{Cat}, {\sc Keller and Lenz}~\cite{KeLe},
{\sc Bendikov and Saloff-Coste}~\cite{BeSa}.
Analogously, the projection $(Y_t)_{t \ge 0}$ on $\R$ evolves like ordinary BM with 
drift as long as it does not visit any integer. When it visits an integer, BM receives
an additional random ``kick'' in the positive or negative direction.

The main goal of this paper is to describe how Brownian motion on $\HT$
evolves spatially. Main tool for this study is the sequence $\bigl(\tau(n)\bigr)$ 
of the stopping times of the successive visits of $(X_t)$ in the 
bifurcation lines $\Lf_v$, 
$v \in V(\T)$ (the vertex set of the tree). The increments 
$\tau(n) - \tau(n-1)$ are i.i.d. for $n \ge 2$, have exponential moments
and an explicitly computable Laplace transform, see \S \ref{sec:rw}. That section contains 
further basic preliminary results. In particular, we study
the distribution of the location of the process at time $\tau(1)$, which is
the law governing the process $(X_{\tau(n)})$. This is quite subtle, because the 
singularities of our structure require care when trying to implement methods 
that appear to be ``obvious'' in the classical smooth setting.  

The state space of the 
induced Markov process $(X_{\tau(n)})_{n \ge 0}$ is the disjoint union of
all bifurcation lines. 
The projection $(Z_{\tau(n)})_{n \ge 0}$ of that process on
$\Hb$ can be interpreted as a random walk on the group $\Aff(\Hb,\qq)$. 
It can be treated via the methods of the work of {\sc Grincevicius} \cite{Gr1},
\cite{Gr2}.
At the same time, the projection $(W_{\tau(n)})_{n \ge 0}$ is a nearest
neighbour random walk on the (vertex set of the) tree whose transition
probabilities are invariant under the action of the affine group of $\T$.
It can also be considered as a random walk on that group. Random walks
of this type were studied in detail by {\sc Cartwright, Kaimanovich and
Woess}~\cite{CKW}. The synthesis of those results on the two affine
groups of $\Hb$ and of $\T$ is crucial for our study.

In \S \ref{sec:boundary}, we consider the natural geometric
boundary at infinity of $\HT$. Since $\HT$ is naturally embedded in the direct
product $\Hb \times \T$, its natural compactification is its closure
in $\wh \Hb \times \wh \T$. The boundary of $\HT$ is the set of points added 
in this way. Here,  $\wh\T$ is the well-known end 
compactification of the tree, while $\wh \Hb$ is the classical compactification
of hyperbolic plane (the closed unit disk in the disk model of $\Hb$,
or equivalently -- in the upper half plane situation -- the upper half plane
together with its bottom line $\R$ and the ``upper'' boundary point at
infinity).  We show that in the topology of that compactification,
Brownian motion on $\HT$ converges almost surely to a limit random variable
that lives on the boundary. In general, we can get quite good information about
the law of that limit random variable, but it can be given explicitly only
in special cases regarding the choice of 
the parameters $\alpha$, $\beta$. 
Convergence to the boundary goes hand in hand with computation of 
the linear \emph{rate of escape} $\ell(\alpha,\beta)$, that is,  
$$
\dist_{\HT}(X_t,X_0) / t \to  \ell(\alpha,\beta) \quad\text{almost surely, as}
\; t \to \infty\,.
$$
It is the same as the rate of escape of $(Y_t)$ on $\R$. A basic tool for 
boundary convergence and rate of escape is the the notion of \emph{regular sequences}
of {\sc Kaimanovich}~\cite{Kai}. 

Next, in \S \ref{sec:CLT}, we derive a central limit theorem, concerning convergence in 
law of
$$
\Bigl(\dist_{\HT}(X_t,X_0) - t\,\ell(\alpha,\beta)\Bigr)\Big/\sqrt{t}\,.
$$
When $\ell(\alpha,\beta) > 0$, the limit law is centred 
normal distribution, and we also explain how to compute its variance
$\sigma^2(\alpha,\beta) > 0$. When $\ell(\alpha,\beta) = 0$, the result as well as the limit
distribution are somewhat more complicated. 

The interplay of BM with the boundary provides the bridge to the
potential theoretic part of our work, that will be laid out in forthcoming work
\cite{BSSW3}. 

In concluding the Introduction, we want to underline how similar the geometric 
features as well
as the properties of Brownian motion (resp. random walks) and the associated
harmonic functions are on DL-graphs \& lampligher groups, the $\Sol$-manifold 
(resp. \hbox{-group}) and treebolic space. In spite of the different techniques needed
for each of the three, the realisation of those analogies, as
well as the detailed study undertaken here, have become possible via
the geometric interpretation of those stuctures as horocyclic products. 

On the other hand, as already indicated, the elaboration and use of the 
analytic and probabilistic tools
for this study are quite subtle in view of the singularities of $\HT$ at the 
bifurcation lines, thus providing a first concrete implementation of the 
analysis on strip complexes developed in \cite{BSSW}.

\section{Geometry and isometries of treebolic space}\label{sec:geometry}
We start by describing the relevant features of the homogeneous tree 
$\T=\T_{\pp}$.
Here, we consider $\T$ as a one-complex, where each edge is a copy of the
unit interval $[0\,,\,1]$. The discrete graph metric 
$\dist_{\T}(v_1,v_2)$ on the vertex 
set ($0$-skeleton) $V(\T)$ of $\T$ is the length (number of edges) on the shortest
path between $v_1$ and $v_2$. This metric has an obvious ``linear'' extension to 
the one-skeleton.

We partition the vertex set into countably many sets $H_k\,$, $k \in \Z$, such
that each $H_k$ is countably infinite, and every vertex $v \in H_k$ has
precisely one neighbour $v^-$ (the \emph{predecessor} of $v$) in $H_{k-1}$ and
$\pp$ neighbours in $H_{k+1}$ (the \emph{successors} of $v$), each of with 
has $v$ as its predecessor. See Figure~2. The sets $H_k$ are called \emph{horocycles}.
For $v \in H_k$, we define $\hor(v) = k$. There is also a horocycle $H_t$
for any  real $t$: if $k = \lceil t \rceil$ and $v \in V(\T)$ with
$\hor(v) = k$, then the metric edge $[v^-,v]$ meets $H_t$ precisely in the 
point $w$ which is at distance $k-t$ from $v$, and we set $\hor(w)=t$.

In addition to this basic description, we shall need futher details.
A \emph{geodesic path}, resp. \emph{geodesic ray}, resp. \emph{infinite
geodesic} in $\T$ is the image of an isometric embedding $t \mapsto w_t \in \T$ of 
a finite  interval $[a\,,b]$, resp. one-sided infinite interval $[a\,,\,\infty)$, 
resp. $\R$, that is, $\dist(w_s,w_t) = |t-s|$ for all $s, t$. 
An \emph{end} of $\T$ is an equivalence class of rays, where
two rays $(w_t)$ and $(\bar w_t)$ are equivalent if they coincide up to 
finite initial pieces, i.e., there are $s, t_0 \in \R$ such that
$\bar w_t = w_{s+t}$ for all $t \ge t_0$.
We write $\bd \T$ for the space of ends, and $\wh \T = \T \cup \bd \T$. 
For all $\eta, \zeta \in \wh \T$ there is a unique geodesic $\geo{\eta\,\zeta}$ 
that connects the two. In particular, if $w \in \T$ and $\xi \in \bd \T$ then 
$\geo{w\,\xi}$ is the ray that starts at $w$ and represents $\xi$.
Furthermore, if $\xi, \eta \in \bd \T$ ($\xi \ne \eta$) then
$\geo{\eta\,\xi}$ is the infinite geodesic whose two halves (split at any of
its points) are rays that respresent $\eta$ and $\xi$, respectively.
For $v,w \in \T$, $v \ne w$, we define the cone
$\wh \T(v,w) = \{ \zeta \in \wh \T : w \in \geo{v\,\zeta} \}$.
For $\xi \in \bd\T$, the collection of all cones containing $\xi$ is a 
neighbourhood basis of  $\xi$, while a neighbourhood basis of $w \in \T$
is given by all open balls in the tree metric. Thus, we obtain a topology wich 
makes $\wh \T$ a compact Hausdorff space
with the vertex set $V(\T)$ as a discrete subset and $\bd \T$
a totally disconnected, compact subset. \\  
$$
\beginpicture 

\setcoordinatesystem units <.7mm,1mm>

\setplotarea x from -10 to 104, y from -84 to -3

\arrow <6pt> [.2,.67] from 0 0 to 80 -80

\plot 32 -32 62 -2 /

 \plot 16 -16 30 -2 /

 \plot 48 -16 34 -2 /

 \plot 8 -8 14 -2 /

 \plot 24 -8 18 -2 /

 \plot 40 -8 46 -2 /

 \plot 56 -8 50 -2 /

 \plot 4 -4 6 -2 /

 \plot 12 -4 10 -2 /

 \plot 20 -4 22 -2 /

 \plot 28 -4 26 -2 /

 \plot 36 -4 38 -2 /

 \plot 44 -4 42 -2 /

 \plot 52 -4 54 -2 /

 \plot 60 -4 58 -2 /

 \plot 99 -29 64 -64 /

 \plot 66 -2 96 -32 /

 \plot 70 -2 68 -4 /

 \plot 74 -2 76 -4 /

 \plot 78 -2 72 -8 /

 \plot 82 -2 88 -8 /

 \plot 86 -2 84 -4 /

 \plot 90 -2 92 -4 /

 \plot 94 -2 80 -16 /


\setdots <3pt>
\putrule from -4.8 -4 to 102 -4
\putrule from -4.5 -8 to 102 -8
\putrule from -1.3 -16 to 102 -16
\putrule from -1.0 -32 to 102 -32
\putrule from  2.2 -42 to 102 -42
\putrule from -1.0 -64 to 102 -64
\setdashes <3pt>
\linethickness =.7pt
\putrule from -3 6 to 102 6
\setlinear

\put {$\vdots$} at 32 3
\put {$\vdots$} at 64 3

\put {$\dots$} [l] at 103 -6
\put {$\dots$} [l] at 103 -48

\put {$H_{-3}$} [l] at -12.5 -64
\put {$H_{-2.3}$} [l] at -12.5 -42
\put {$H_{-2}$} [l] at -12.5 -32
\put {$H_{-1}$} [l] at -12.5 -16
\put {$H_0$} [l] at -12.5 -8
\put {$H_1$} [l] at -12.5 -4

\put {$\vdots$} at -10 3
\put {$\vdots$} [B] at -10 -70
\put {$\om$} at 82 -82

\put {$\bd^*\T$} [l] at -14 6

\put {\scriptsize $\bullet$} at 8 -8
\put {$o$} [rt] at 7.6 -8.8
\put {\scriptsize $\bullet$} at 96 -32
\put {$v$} [rt] at 100.4 -33
\put {\scriptsize $\bullet$} at 64 -64
\put {$v^-$} [rt] at 63.8 -64.6
\endpicture
$$

\vspace{.1cm}

\begin{center}
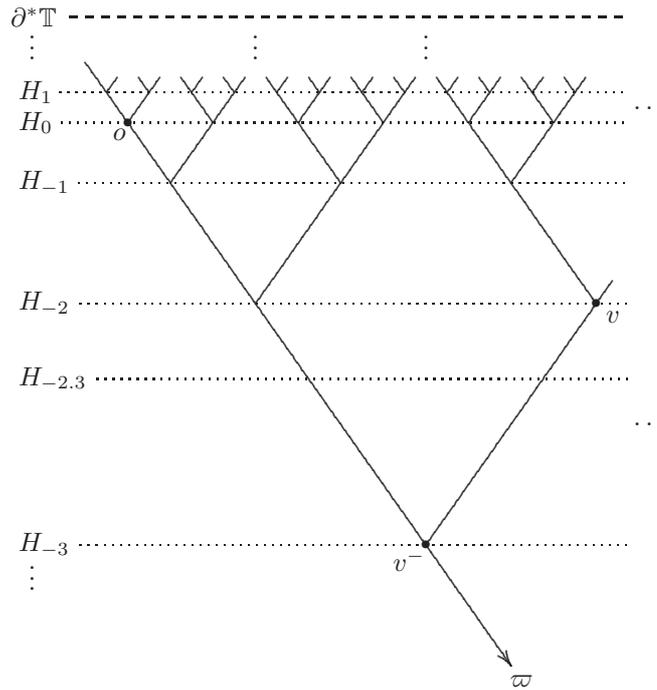
{\sl Figure~2.} The ``upper half plane'' drawing of $\T_2$ 
(top down)\footnote{This figure also appears in \cite{BSSW}.}
 
\end{center}

\medskip
We choose and fix a reference vertex (root) $o \in H_0$. The geodesic ray whose
vertices consist of the root and all its ancestors (= iterated predecessors)
defines a reference end $\om \in \bd \T$, the lower boundary point in Figure~2.
We set $\bd^* \T = \bd \T \setminus \{\om\}$, the upper boundary in Figure~2. 
For $w_1,w_2 \in \wh \T \setminus \{ \om \}$, their confluent (or maximal 
common ancestor)
$b = w_1 \cf w_2$ with respect to $\om$ is defined by
$\geo{w_1\,\om} \cap \geo{w_2\,\om} = \geo{b\,\om}$.
The function  $\hor: \T \to \R$ defined above is the 
\emph{Busemann function} of $\T$ with respect to $\om$, 
which can be written as
\begin{equation}\label{eq:hor}
\hor(w) = \dist(w,w \cf o) - \dist(o,w \cf o).
\end{equation}
In addition, we note that
\begin{equation}\label{eq:tree-conf}
\dist_{\T}(v,w) = \dist_{\T}(v\,,v\cf w) + \dist_{\T}(v\cf w,w)
= \hor(v)+\hor(w) - 2\hor(v\cf w).
\end{equation}
There is a natural
Lebesgue measure $dw$ on $\T$, which on each edge is a copy of
standard Lebesgue measure on the unit interval. 

The natural compactification $\wh \Hb$ of the hyperbolic plane $\Hb$ is the closed 
unit disk, when we
use the Poincar\'e disk model. In our upper half plane model, $\wh \Hb$ is the closed
upper half plane together with the the point at infinity $\binfty$. The
corresponding boundary  $\partial \Hb$ consists of $\binfty$ together with the
lower boundary line $\partial^*\Hb = \R$. The Busemann function on $\Hb$
with respect to $\binfty$ is $z \mapsto \log (\IM z)$, where $\IM z$ is the imaginary part 
of $z$. For $z, z' \in \wh\Hb\setminus \{\binfty\}$, we can define 
the hyperbolic analogue $z \wedge z'$ of the confluent: 
$z \wedge z = z$, and when $z \ne z'$, then $z \wedge z'$ is the point
on the (hyperbolic) geodesic $\geo{z\,z'}$ with maximal imaginary part.
Recall that $\geo{z\,z'}$ is part of a circle centred on $\R$ which is orthogonal
to that boundary line. The function $(z,z') \mapsto z \wedge z'$
is continuous from 
$(\wh \Hb \setminus \{\binfty\}) \times (\wh \Hb \setminus \{\binfty\})$
to $\wh \Hb \setminus \{\binfty\}$. (The analogous property holds for the
tree.)
Similarly to \eqref{eq:tree-conf}, we have for $z, z' \in \Hb$
\begin{equation}\label{eq:hyp-conf}
\begin{gathered}
\dist_{\Hb}(z,z') = \dist_{\Hb}(z\,,z\wedge z') + \dist_{\Hb}(z\wedge z', z')
\AND\\ 
\Big|\dist_{\Hb}(z,z')  \,-\, \Bigl(2 \log (\IM\, z \wedge z') - \log(\IM z)
- \log (\IM z')\Bigr) \Big| \le \log 4.
\end{gathered}
\end{equation}
We are not sure whether the last inequality appears in the literature
very often; its proof is an amusing exercise of handling the hyperbolic metric.

In the same way as $\T$ is subdivided horizontally by the horocycles $H_k$, 
$k \in \Z$, we subdivide $\Hb$ into the horizontal strips $\Sf_k$ delimited by 
the lines $L_k$ consisting of all $x + \im y \in \Hb$ with $y=\qq^k$, 
$k \in \Z$, see Figure~3. 
Note that all $\Sf_k$ are hyperbolically isometric.   

As outlined in introduction and abstract, treebolic space with parameters $\qq$ and
$\pp$ is 
\begin{equation}\label{eq:treebolicdef}
\HT(\qq,\pp) = \{ \zf = (z,w) \in \Hb \times \T_{\pp} : 
\hor(w) = \log_{\qq}(\IM z) \}\,.
\end{equation}
$$
\beginpicture 

\setcoordinatesystem units <1.4mm,1.00mm>

\setplotarea x from -40 to 40, y from 0 to 80

\arrow <6pt> [.2,.67] from 0 0 to 0 77.5

\plot -40 4  40 4 /

\plot -40 8  40 8 /

\plot -40 16  40 16 /

\plot -40 32  40 32 /

\plot -40 64  40 64 /

\put {$\im$} [rb] at -0.3 8.6
\put {$\scs \bullet$} at 0 8

\put {$y = \qq^{-1}$} [r] at -40 4
\put {$y = 1$} [r] at -43 8
\put {$y = \qq$} [r] at -43 16
\put {$y = \qq^2$} [r] at -42 32
\put {$y = \qq^3$} [r] at -42 64

\put {$\binfty$} [t] at 0 80

\put {$\R$} [r] at -45 0

\put {$\Sf_1$} at 30 12
\put {$\Sf_2$} at 30 24
\put {$\Sf_3$} at 30 48

\setdashes <3pt>
\linethickness =.7pt
\putrule from -43.4 0 to 40.4 0 
\setlinear

\endpicture
$$
\vspace{.1cm}

\begin{center}
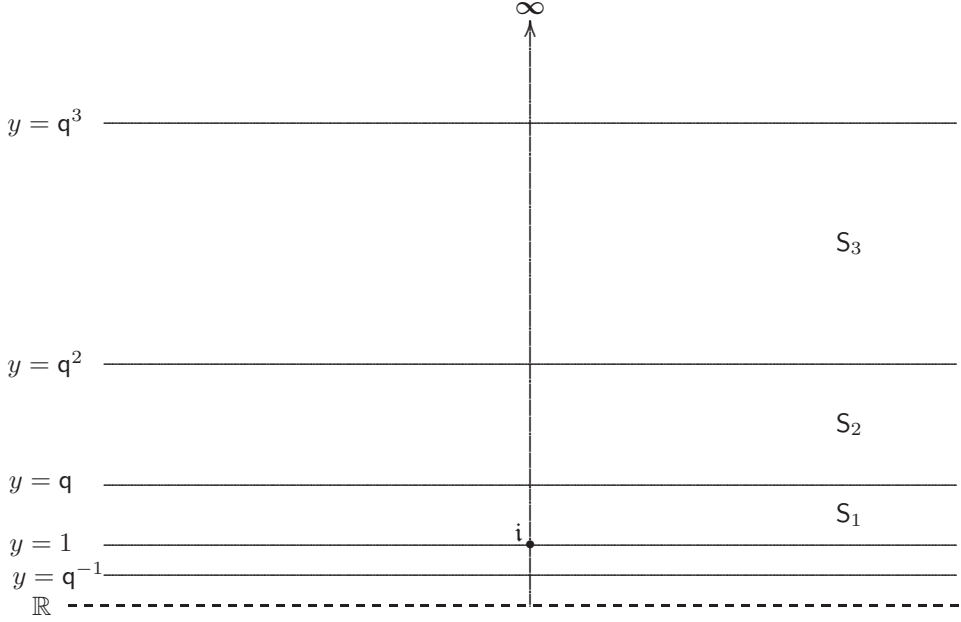
{\sl Figure~3.} Hyperbolic upper half plane $\Hb$ subdivided in 
isometric strips\,\footnote{This figure also appears in \cite{BSSW}.}
\end{center}

\bigskip

Thus, Figures~2 and~3 are the 
``side'' and ``front'' views of $\HT$, that is, the images of $\HT$
under the projections $\pi_{\Hb}: (z,w) \mapsto z$ and $\pi_{\T}: (z,w) \mapsto w$, 
respectively. 
For a vertex $v \in V(\T)$, let 
\begin{equation}\label{eq:line}
\Lf_v = \{ \zf = (z,v) \in \HT : \IM z = \qq^{\hor(v)} \} 
= L_{\hor(v)} \times \{v\}\,.
\end{equation}
Then $\Lf_v$ and $\Lf_{v^-}$ are the upper and lower lines (respectively) 
in $\HT$ that delimit the strip 
\begin{equation}\label{eq:strip}
\Sf_v = \{ \zf = (z,w) \in \HT : w \in [v^-,v] \}\,.
\end{equation}
Here, $w  \in [v^-,v]$ is an element of the edge $[v^-,v]$ of $\T$, which
is (recall) a copy of the unit interval. For $\zf = (z,w) \in \HT$, we shall
sometimes write $\RE \zf = \RE z$ for the real part of $z$.

For each end $\xi \in \bd^*\T$, treebolic space contains the isometric copy
$$
\Hb_{\xi} = \{ \zf = (z,w) \in \HT : w \in \geo{\xi\,\om}\}
$$ 
of $\Hb$, and if $\xi,\eta \in \bd^*\T$ are distinct and $v=\xi \cf \eta$ (a vertex),
then $\Hb_{\xi}$ and $\Hb_{\eta}$ ramify along the line $\Lf_v$, that is,
$\Hb_{\xi} \cap \Hb_{\eta} = \{ (z,w) \in \HT : w \in \geo{v\,\om}\}$.

The metric of $\HT$ is given by the hyperbolic length element in
the interior of each strip. Its natural geodesic continuation is given 
as follows:
consider two points $\zf_1=(z_1,w_1) , \zf_2=(z_2,w_2) \in \HT$. Let $d_{\Hb}(z_1,z_2)$ by the hyperbolic distance
between $z_1$ and $z_2$, and let $v=w_1 \cf w_2\,$. Then 
\begin{equation}\label{eq:metric}
\dist_{\HT}\bigl(\zf_1,\zf_2) = 
\begin{cases} \dist_{\Hb}(z_1,z_2)\,,&\text{if}\; v \in \{ w_1, w_2 \},\\
  \min \{ \dist_{\Hb}(z_1,z)+\dist_{\Hb}(z,z_2) : z \in L_{\hor(v)} \}\,,&\text{if}\; 
        v \notin \{ w_1, w_2 \}.
\end{cases}
\end{equation}
Indeed, in the first case, $\zf_1$ and $\zf_2$ belong to a common copy 
$\Hb_{\xi}$ of $\Hb$. In the second case, $v \in V(\T)$, and there are 
$\xi_1, \xi_2 \in \bd^*\T$  such that $\xi_1 \cf \xi_2 =v$ and 
$\zf_i \in \Hb_{\xi_i}$ lie above the line $\Lf_v$, 
so that it is necessary to pass through some point $\zf = (z, v) \in \Lf_v$ 
on the way from $\zf_1$ to $\zf_2$. See Figure~4.
$$
\beginpicture 

\setcoordinatesystem units <1.1mm,1.1mm>

\setplotarea x from -16 to 54, y from -4 to 30

\plot 0 0  40 0  55 26  15 26  0 0  -18 23  13 23   /

\put{$\Lf_v$} [l] at 41.5 0
\put{\scriptsize $\bullet$} at 34 16
\put{$\zf_2$} [l] at 35 16
\put{\scriptsize $\bullet$} at -2 12
\put{$\zf_1$} [r] at -2.8 12
\put{\scriptsize $\bullet$} at 20.5 0
\put{$\zf$} [t] at 20.5 -1

\circulararc -10 degrees from -2 12 center at -10 -30

\circulararc 19 degrees from 34 16 center at 75 -33

\setdashes <2pt>
\circulararc -24 degrees from 5.4 9.9 center at -10 -30

\endpicture
$$
\begin{center}
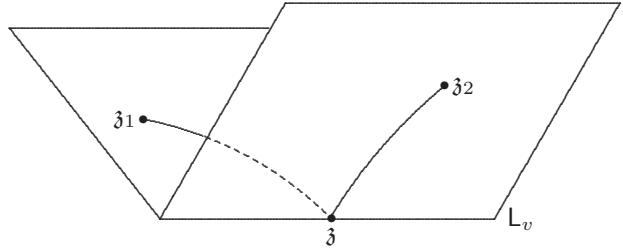
{\sl Figure~4.\/} Geodesic connecting $\zf_1$ and $\zf_2$ in $\HT$.
\end{center}

\medskip
\begin{pro}\label{pro:metric} For $\zf_1=(z_1,w_1)$, 
$\zf_2=(z_2,w_2) \in \HT$,
with $\delta = \log(1+\sqrt 2)$,
$$
\begin{aligned}
\dist_{\HT}(\zf_1, \zf_2) \le 
\dist_{\Hb}(z_1,z_2) &+ (\log \qq)\, \dist_{\T}(w_1,w_2) \\
&- \underbrace{(\log \qq) \, |\hor(w_1) - \hor(w_2)|}_{\displaystyle
|\log\IM z_1 - \log\IM z_2|}
\le \dist_{\HT}(\zf_1, \zf_2) +
2\delta\,.
\end{aligned}
$$
\end{pro}

\begin{proof}
In the first case of \eqref{eq:metric}, we have $\dist_{\HT}(\zf_1, \zf_2) = 
\dist_{\Hb}(z_1,z_2)$ and $\dist_{\T}(w_1,w_2) = |\hor(w_1) - \hor(w_2)|$.
Therefore, the left hand side inequality of the proposition is indeed an
equality.

\smallskip

We consider the second case of \eqref{eq:metric}. 
We suppose without loss of generality that 
$\IM z_1 \le \IM z_2$ and $\RE z_1 \le \RE z_2$. 
Set $k = \hor(v)$, and let $z \in \Lf_k$ be a point that
realizes the minimum in \eqref{eq:metric}, corresponding to $\zf  = (z,v)$ in
Figure~4. On the vertical ray in $\Hb$ going upwards from 
$z$ to $\binfty$, let $z_1'$ be the point with $\IM z_1' = \IM z_1\,$.
Also, we let $z'$ be the point on $L_k$ with $\RE z' = \RE z_1\,$.
See Figure~5, showing the respective points and geodesic arcs in $\Hb$.

We start with the left hand one of the two proposed inequalities,
and use the minimising property of $z$: the
distance $\dist_{\HT}(\zf_1,\zf_2)$ is bounded above by the length of
any path in $\Hb$ that starts at $z_1\,$, ends at $z_2$ and visits
$\Lf_k$ in between. We choose the following path: we first move vertically
from $z_1$ down to $z'$, then back up to $z_1$, and then along the geodesic 
arc from $z_1$ to $z_2\,$. 
The length of this path is  
$$
2(\log \IM z_1 - \log \IM z') + \dist_{\Hb}(z_1,z_2)\,,
$$
which coincides with the middle term of our double inequality.   

$$
\beginpicture 

\setcoordinatesystem units <.8mm,.8mm>

\setplotarea x from -30 to 46, y from -4 to 45

\plot -30 0  46 0 /

\put{$\Lf_k$} [l] at 47 0
\put{\scriptsize $\bullet$} at 0 0
\put{$z$} [t] at 0 -1.8
\put{\scriptsize $\bullet$} at -26 26
\put{$z_1$} [r] at -27 26
\put{\scriptsize $\bullet$} at 45 27
\put{$z_2$} [l] at 46 27
\put{\scriptsize $\bullet$} at 0 26
\put{$z_1'$} [l] at 1 26
\put{\scriptsize $\bullet$} at -26 0
\put{$z'$} [t] at -26 -1

\circulararc -90 degrees from 0 0 center at 36 -9 

\circulararc 61 degrees from 0 0 center at -35 -9

\circulararc 90 degrees from 45 27 center at 10 -9
\setdashes <2pt>

\plot 0 0  0 40.3 /
\plot -26 0  -26 26 /

\endpicture
$$

\begin{center}
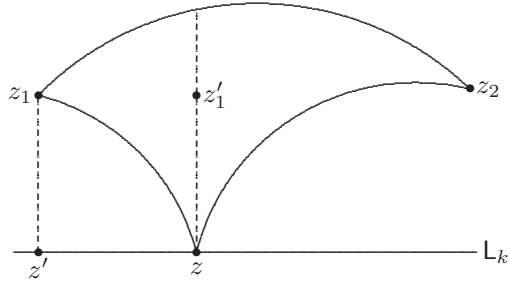
{\sl Figure~5.\/} The geodesic triangle in $\Hb$ formed by the 
projections of $\zf_1\,$, $\zf_2$ and $\zf$.
\end{center}

\medskip

We now consider the right hand one of the two proposed inequalities. 
By \eqref{eq:metric}, 
$\dist_{\HT}(\zf_1,\zf_2) = \dist_{\Hb}(z_1,z) +  \dist_{\Hb}(z_2,z)$.
On the other hand, 
$$
\begin{aligned} 
(\log \qq)\, \dist_{\T}(w_1,w_2) 
&= (\log \IM z_1 - \log \IM z) +  (\log \IM z_2 - \log \IM z)  
\AND\\ 
(\log \qq) \, |\hor(w_1) - \hor(w_2)| &= \log \IM z_2 - \log \IM z_1\,,
\end{aligned}
$$
because the last term was assumed to be $\ge 0$. Thus,
$$
(\log \qq)\, \dist_{\T}(w_1,w_2) 
- (\log \qq) \, |\hor(w_1) - \hor(w_2)| = 2 \dist_{\Hb}(z_1',z)\,.
$$
Thus, the claim of the proposition is equivalent with
$$
\dist_{\Hb}(z_1',z) - 2\delta \le
\frac12\Bigl( \dist_{\Hb}(z_1,z) + \dist_{\Hb}(z_2,z) - \dist_{\Hb}(z_1,z_2)\Bigr)
\le \dist_{\Hb}(z_1',z)\,,
$$ 
and we still need to proof the left hand inequality.
Now recall that $\Hb$ is the classical model of a geodesic metric space which 
is Gromov-hyperbolic: for the given value of $\de \ge 0$, every geodesic
triangle is $\de$-thin. (That is, for any point on one of the three sides,
there is a point at distance at most $\de$ on the union of the
other two sides.) We refer to {\sc Gromov}~\cite{Gro}, 
{\sc Coornaert, Delzant and Papadopoulos}~\cite{CDP} and/or 
{\sc Ghys and de la Harpe}~\cite{GH} for all details regarding
hyperbolic metrics and spaces. 
Now, $\frac12\bigl( \dist(z_1,z) + \dist(z_2,z) - \dist(z_1,z_2)\bigr) =
(z_1|z_2)_z$
is just the so-called \emph{Gromov product} of $z_1$ and $z_2$ with respect to 
the reference point $z$. It is well known that the Gromov product on any
geodesic hyperbolic metric space satisfies
$$
\dist(z, \geo{z_1\,z_2}) -2\de \le (z_1|z_2)_z  
\le \dist(z, \geo{z_1\,z_2})\,,
$$
where $\geo{z_1\,z_2}$ is of course the geodesic arc between
$z_1$ and $z_2\,$. In our situation, we have 
$\dist_{\Hb}(z, \geo{z_1\,z_2}) \ge \dist_{\Hb}(z_1',z)$, and
the desired inequality follows. 
\end{proof}

Proposition \ref{pro:metric} should be compared with the formula of 
\cite[Prop. 3.1]{Ber} for the graph metric of the $\DL$-graphs, which is 
of the same form (without the $\delta$).
Note that the width of a strip $\Sf_v$ is
$\dist_{\HT}(\Lf_{v^-},\Lf_v) = \log \qq$, while its image under the projetction
$\pi_{\T}$ is the (metric) edge $[v^-,v]$, which has length $1$. That is, 
in the construction of $\HT$ from $\T$ and $\Hb$, the tree is
stretched by a factor of $\log \qq$.

Also note that the coordinates $(z,w)$ of $\HT$ used in 
\eqref{eq:treebolicdef} are useful in order to see the nature of $\HT$ as a
horocylic product and for deducing algebraic--geometric properties. However,
by their nature, these coordinates are not independent. The resulting
redundancy can be avoided by yet another description, more suitable for analytic
purposes; see \cite[\S 2.B]{BSSW}. It is not used here in order to avoid abundance of
multiple notation.

\smallskip

The \emph{area element} of $\HT$ is $d\zf = y^{-2} dx\,dy$ for $\zf=(z,w)$
in the interior of every $\Sf_v$, where $z=x+\im y$ and $dx$, $dy$ are 
Lebesgue measure: this is (a copy of) the standard hyperbolic upper half
plane area element. The area of the lines $\Lf_v$ is of course $0$. 

\begin{dfn}\label{dfn:fv}
For a real function $f$ on $\HT$, we write $f_v$ for its restriction to the closed
strip $\Sf_v\,$, 
where $v \in V(\T)$. For its values, we write $f_v(z) = f(z,w)$, 
where $w \in \T$  is the unique element on the edge $[v^-\,,v]$ such that
$(z,w) \in \HT$ (that is, $\hor(w) = \log_{\qq}(\IM z)$).  

Analogous notation is used for the restriction of a function $f$ defined on
$\Omega \subset \HT$ to $\Omega \cap \Sf_v\,$.
\end{dfn}

While we think 
of $f_v$ as a function on $\Sf_v\,$, it is formally a function defined
for complex $z = x+\im y \in \Sf_{\hor(v)} \subset \Hb$.
The integral of $f$ with respect to the area element is given by
\begin{equation}\label{eq:measure}
\int_{\HT} f(\zf) \,d\zf = 
\sum_{v \in V(\T)} \int_{\Sf_{\hor(v)}} f_v(x + \im y)\,y^{-2}\,dx\,dy\,,
\end{equation}
whenever this is well defined in the sense of Lebesgue integration.

Next, we determine the isometry group of $\HT(\qq,\pp)$ and its modular 
function. Recall here that the modular function $\md$ of an arbitrary locally compact group 
$G$ is the continuous homomorphism from the group into the multiplicative 
group $\R_+$ with the property that for left Haar measure $dg$ on $G$, one has
$$
\int_G f(gg_0)\,dg = \md(g_0)^{-1} \int_G f(g)\,dg
$$
for every integrable function $f$ on $G$.

Consider the action on $\Hb$ of the group of affine transformations
\begin{equation}\label{eq:AffRq}
\Aff(\Hb,\qq) = \left\{ g=\Bigl(\!\begin{smallmatrix} \qq^n & b \\[2pt] 0 & 1 \end{smallmatrix}\!\Bigr) :
n \in \Z\,,\ b\in \R \right\}, \quad
\text{where} \;\; gz = \qq^n z + b\,,\; z \in \Hb\,.
\end{equation}
Thus, $g_1g_2 = \bigl(\begin{smallmatrix}\qq^{n_1+n_2} & \,b_1 + \qq^{n_1}b_2\\ 0 & 1
                      \end{smallmatrix}\bigr)$
for $g_i=\bigl(\begin{smallmatrix}\qq^{n_i} &b_i\\ 0 & 1\end{smallmatrix}\bigr)$, $i=1,2$.
This group acts by isometries on $\Hb$ and leaves the set of lines $y = \qq^k$, 
$k \in \Z$, invariant. The full group of isometries of $\Hb$ with the latter property 
is generated by $\Aff(\Hb,\qq)$ and the reflection along the $y$-axis. 
Our group is locally compact, and left Haar measure $dg$ and its modular function 
$\md_{\Hb} = \md_{\Hb,\qq}$ are given 
by
\begin{equation}\label{eq:modularAffRq}
dg = \qq^{-n}\,dn\,db \AND \md_{\Hb}(g) = \qq^{-n}\,,\quad\text{if}\quad
g=\Bigl(\!\begin{smallmatrix}\qq^n& b \\[2pt] 0 & 1 \end{smallmatrix}\!\Bigr)\,.
\end{equation}
Here, $dn$ is counting measure on $\Z$ and $db$ is Lebesgue measure on $\R$. 

Regarding the tree, first note that every isometry is the natural linear extension
of an automorphism, that is, a neighbourhood preserving permutation of the vertex
set $V(\T)$. Also, note that the action of each isometry extends continuously to
$\wh\T$, since isometries send geodesic rays to geodesic rays and preserve their
equivalence. Let $\Aut(\T_{\pp})$ denote the full isometry group of $\T_{\pp}$.
Following {\sc Cartwright, Kaimanovich and Woess}~\cite{CKW},
the \emph{affine group} of $\T_{\pp}$ is
\begin{equation}\label{eq:AffTp}
\Aff(\T_{\pp}) = \{ \ga \in \Aut(\T_{\pp}) : \ga\om=\om \}\,. 
\end{equation}
This is a locally compact, totally disconnected and compactly generated 
group with respect to the topology of pointwise convergence, and it acts
transitively on $V(\T)$.
The name is chosen (1)~because of the analogy with the classical affine group which
is just the group of (orientation preserving) isometries of $\Hb$ that fix the
boundary point $\binfty$, and (2)~because the affine group over any local field whose
residual field has order $\pp$ embeds naturally into $\Aff(\T_{\pp})$, see
\cite{CKW} and below. The elements $\ga$ of $\Aff(\T)$ are also characterized 
by the
property $\ga(v^-) = (\ga v)^-$ for every $v \in V(\T)$, or equivalently, by
$\ga(w_1 \cf w_2) = (\ga w_1) \cf (\ga w_2)$ for all $w_i \in \T$. 
Consequently, the mapping $\Phi: \Aff(\T) \to \Z$ defined by 
$\ga \mapsto \hor(\ga w) - \hor(w)$ is independent of $w \in \T$ and a homomorphism. 
Thus $\ga(H_t) = H_{t+k}$ if $\hor(\ga w) - \hor(w)=k$.
As a matter of fact, this mapping appears in the modular function 
$\md_{\T} = \md_{\T_{\pp}}$ of $\Aff(\T_{\pp})$, see \cite{CKW}:
\begin{equation}\label{eq:modularAffTp}
\md_{\T}(\ga) = \pp^{\Phi(\ga)} \quad \text{where}\quad 
\Phi(\ga)=\hor(\ga w) - \hor(w) \,,\quad\text{if} 
\quad \ga \in \Aff(\T_{\pp})\,,\;w \in \T\,.
\end{equation}

In the following theorem, we collect several rather 
straightforward properties of the isometry group of $\HT(\qq,\pp)$.

\begin{thm}\label{thm:isogroup} The group
$$
\Af = \Af(\qq,\pp) = \{ (g,\ga) \in \Aff(\Hb,\qq) \times \Aff(\T_{\pp}) : 
\log_{\qq} \md_{\Hb}(g) + \log_{\pp} \md_{\T}(\ga) = 0 \} 
$$
acts on $\HT(\qq,\pp)$ by isometries $(g,\ga)(z,w) = (gz,\ga w)$. It is
the semidirect product
$$
\Af = \R \rtimes \Aff(\T) \quad \text{with respect to the action}
\quad b \mapsto \qq^{\Phi(\ga)}\,b\,, \; \ga \in \Aff(\T)\,,\; b \in \R\,.
$$  
The full group
of isometries of $\HT(\qq,\pp)$ is generated by $\Af(\qq,\pp)$ and the reflection
$$
{\mathfrak s}(x+\im y,w) = (-x+\im y,w)\,.
$$
It acts on $\HT(\qq,\pp)$ with compact quotient isomorphic with the circle of length
$\log \qq$, and it leaves the area element of $\HT$ invariant. 

As a closed subgroup of $\Aff(\Hb,\qq) \times \Aff(\T_{\pp})$, the group $\Af$ 
is locally compact, compactly generated and amenable, and its modular function is 
given by
$$
\md_{\Af}(g,\ga)= (\pp/\qq)^{\Phi(\ga)}\,. 
$$
\end{thm}

\begin{proof} (1) Let $(g,\ga) \in \Af$ and $(z,w) \in \HT$, with 
$g=\bigl(\begin{smallmatrix} \qq^n & b \\ 0 & 1 \end{smallmatrix}\bigr)$ and 
$z=x+\im y$. Then $gz = (b+\qq^n x) + \im (\qq^n y)$, and $\hor(\ga w) = \hor(w)+n$.
Thus, $\hor(w) = \log_q(\IM z)$ implies $\hor(\ga w) = \log_q(\IM gz)$, whence
$(gz,\ga w) \in \HT$. From \eqref{eq:metric}, one sees that this is an isometry.
Indeed, let $(z_i,w_i) \in \HT$ ($i=1,2$) and $v = w_1 \cf w_2$.
Then $\ga v = \ga w_1 \cf \ga w_2$. So, if $v \in \{ w_1, w_2 \}$ then
$\ga v \in \{ \ga w_1, \ga w_2 \}$ and 
$$
d\bigl((gz_1,\ga w_1),(gz_2,\ga w_2)\bigr) = d_{\Hb}(gz_1,gz_2) =  
d_{\Hb}(z_1,z_2) = d\bigl((z_1,w_1),(z_2,w_2)\bigr)\,.
$$
If $ v \notin \{ w_1, w_2  \}$ then $v \in V(\T)$ and $\ga \Lf_v = \Lf_{\ga v}$.
If $z_0$ minimizes $d_{\Hb}(z_1,z) + d_{\Hb}(z,z_2)$ among all $z \in L_{\hor(v)}$,
then $gz_0$ minimizes $d_{\Hb}(gz_1,\tilde z) + d_{\Hb}(\tilde z,z_2)$ among all 
$\tilde z \in L_{\hor(\ga v)}$. Thus, 
$d\bigl((gz_1,\ga w_1),(gz_2,\ga w_2)\bigr) = d\bigl((z_1,w_1),(z_2,w_2)\bigr)$
in this case as well. Thus, $\Af$ acts by isometries.

\smallskip

(2) We can identify each element 
$\gf=(g,\ga) \in \Af$ with the pair $[b,\ga] \in \R \times \Aff(\T)$, where 
$g=\bigl(\begin{smallmatrix} \qq^{\Phi(\ga)} & b \\ 0 & 1 \end{smallmatrix}\bigr)$ 
as an affine mapping. It is immediate that with  this identification,
$\Af = \R \rtimes \Aff(\T)$ with the proposed action of $\Aff(\T)$ on $\R$,
namely, the group operation is $[b_1\,,\ga_1][b_2\,,\ga_1] = [b_1 + \qq^{\Phi(\ga_1)}b_2\,, \ga_1\ga_2]$.
 
\smallskip

(3) Let $\gf$ be an isometry of $\HT$. Then it is clear that $\gf$ sends
each line $\Lf_v$ to some other line $\Lf_{\tilde v}$; compare with \cite{FaMo}.
Thus, there is some 
$\ga \in \Aut(\T)$ such that $\gf \Lf_v = \Lf_{\ga v}$ for every $v \in V(\T)$.
We claim that $\ga \in \Aff(\T)$, that is, $\ga v^- = (\ga v)^-$ for all $v \in V(\T)$.

For $v \in V(\T)$, let $\tilde v = \ga v$.
Suppose that $\ga v^- \ne \tilde v^-$. Then $\tilde u_1 = \ga v^-$ must be a successor
of $\tilde v$, that is, $\tilde u_1^- = \tilde v$. Also, since $\pp \ge 2$,
there must be some successor $u$ of $v$ ($u^- = v$) such that $\tilde u_2 = \ga u$
is a successor of $\tilde v$. 
Then $\gf$ maps $\Sf_u \cup \Sf_v$ isometrically to $\Sf_{u_1} \cup \Sf_{u_2}$.
Now, writing $k=\hor(v)$ and $\tilde k=\hor(\ga v)$, we have that
$\Sf_u \cup \Sf_v$ is an isometric copy of $\Sf_{k-1}\cup \Sf_k \subset \Hb$, while 
$\Sf_{u_1} \cup \Sf_{u_2}$ consists of two copies of $\Sf_{k-1}$ glued together along
their bottom line. With the metric \eqref{eq:metric}, these two pieces are 
\emph{not}\ isometric. Thus, it must be $\ga v^- =  (\ga v)^-$, and $\ga \in \Aff(\T)$.
We now also see that $\gf \Sf_v = \Sf_{\ga v}$ for all $v \in V(\T)$. 

Now consider an end $\xi\in \bd^*\T$. It follows from the above that
$\gf \Hb_{\xi} = \Hb_{\ga \xi}$.
Thus, there must be an isometry $g_{\xi}$ of $\Hb$ such that for $(z,w) \in \Hb_{\xi}$,
$\gf(z,w) = (g_{\xi} z, \ga w)$. Then $g_{\xi}$ must be either a M\"obius transformation
or a M\"obius transformation followed by reflection along the $y$-axis.
Since $\gf \Lf_v = \Lf_{\ga v}$ for each $v \in V(\T) \cap \geo{\xi\,\om}$, 
in both of the last cases, that M\"obius transformation is in $\Aff(\Hb,\qq)$. 
Now let $\eta \in \bd^*\T \setminus \{\xi\}$ and set $v = \xi \cf \eta$. 
Then $\Hb_{\xi}$ and $\Hb_{\eta}$ coincide below (and including) the line 
$\Lf_v \subset \HT$, whence $g_{\xi}$ and $g_{\eta}$ coincide below the 
line $L_{\hor(v)}$. But this implies that $g_{\xi} = g_{\eta}=:g$ for all
$\xi, \eta \in \bd^*\T$. Every $(z,w) \in \HT$ lies in $\Hb_{\xi}$ for some
$\xi \in \bd^*\T$. Therefore $\gf(z,w) = (gz,\ga w)$ for all $(z,w) \in \HT$.
This means that either $\gf \in \Af$ (when $g$ itself is a M\"obius transformation),
or ${\mathfrak s}\gf \in \Af$ (when $g$ is a M\"obius transformation followed by 
reflection along the $y$-axis). 

The statement abut the co-compact action and factor space is obvious, and 
it is straightforward that the action of $\Af$, as well as ${\mathfrak s}$,
preserve the area element of $\HT$.

\smallskip

(4) We compute the modular function of $\Af$.
Let $db$ be Lebesgue measure on $\R$ and $d\ga$ left Haar measure on $\Aff(\T)$.
It will be useful to normalise $d\ga$ such that 
\begin{equation}\label{eq:stab}
\int_{\Aff(\T)} \uno_{\Stab(o)}(\ga)\, d\ga =1\,,\quad \text{where}
\quad \Stab(x) = \{ \ga \in \Aff(\T) : \ga o = o\}
\end{equation} is the stabiliser of $o$. It is an open-compact subgroup of
$\Aff(\T)$.
It is a straightforward exercise that in 
the $[b,\ga]$-coordinates of the semidirect product, left Haar measure on $\Af$ is 
given by
\begin{equation}\label{eq:HaarAf}
d\gf = \qq^{-\Phi(\ga)}\, db\,d\ga\,.
\end{equation} 
Now let $\gf_0 = [b_0,\ga_0]$, and let $f \in \CC_c(\HT)$, the space of 
continuous, compactly supported functions. Then, using \eqref{eq:modularAffTp}
$$
\begin{aligned}
\int_{\!\Af}  f(\gf\gf_0) \,d\gf 
&= \int_{\Aff(\T)} \int_{\R} \qq^{-\Phi(\ga)} f[b+\qq^{\Phi(\ga)}b_0, \ga\ga_0]\,db\,d\ga \\
&= \qq^{\Phi(\ga_0)}\! \int_{\R} \int_{\Aff(\T)} \qq^{-\Phi(\ga\ga_0)} 
           f[b, \ga\ga_0]\,d\ga\,db \\
&= \qq^{\Phi(\ga_0)}\! \int_{\R} \md_{\T}(\ga_0)^{-1} \!
        \int_{\Aff(\T)} \!\!\qq^{-\Phi(\ga)} f[b, \ga]\,d\ga\,db 
= (\qq/\pp)^{\Phi(\ga_0)} \!\int_{\!\Af}  f(\gf)\,d\gf\,.
\end{aligned}
$$
This yields $\md_{\Af}(\ga_0) = (\pp/\qq)^{\Phi(\ga_0)}$, as proposed. 
\end{proof}
    
\begin{cor}\label{cor:lattice} When $\pp \ne \qq$, there is no \emph{discrete} 
group that acts on $\HT(\qq,\pp)$ with compact quotient.
\end{cor} 

Indeed, such a group would be a co-compact lattice in the isometry group of
$\HT(\qq,\pp)$, which cannot exist, since the latter group is non-unimodular.
When $\pp=\qq$, the situation is different.

\begin{pro}\label{pro:BS}
The Baumslag-Solitar group $\BS(\pp) = \langle \ab, \bb : \ab\bb=\bb^{\pp}\ab \rangle$
embeds as a co-compact, discrete subgroup into $\Af(\pp,\pp)$.
\end{pro}

\begin{proof} It is well-known that 
\begin{equation}\label{eq:matrix}
\BS(\pp) = \left\{ \Bigl(\!\begin{smallmatrix} \pp^n & \,k/\pp^l \\[2pt] 0 & 1 \end{smallmatrix}\Bigr)
: k, l, n \in \Z \right\}\,.
\end{equation}
In this representation, 
$\ab = \bigl(\begin{smallmatrix} \pp  & 0\\[1pt] 0 & 1\end{smallmatrix}\bigr)$
and $\bb = \bigl(\begin{smallmatrix} 1  & 1\\[1pt] 0 &
1\end{smallmatrix}\bigr)$.
Thus, it is immediate that $\BS(\pp)$ is a (non-discrete) subgroup of
$\Aff(\Hb,\pp)$.

\smallskip

We now explain how our group acts on $\T = \T_{\pp}\,,$ compare e.g. with
\cite[\S 4.A]{CKW}. The \emph{ring} $\Q_{\pp}$ of $\pp$-adic numbers 
consists of all Laurent series in powers of $\pp$ of the form
\begin{equation}\label{eq:padic}
\uf = \sum_{k=m}^{\infty} a_k\,\pp^k\,,\quad m \in \Z\,,\;
a_k \in \{0, \dots, \pp-1\}\,.
\end{equation}
If $a_m \ne 0$ in \eqref{eq:padic}, then we set $|\uf|_{\pp} = \pp^{-m}$,
the $\pp$-adic norm of $\uf$. In addition, we get the neutral element
$0$  of $\Q_{\pp}$ when $a_k=0$ for all $k$ in \eqref{eq:padic}.
Of course, $|0|_{\pp} = 0$.
Addition and multiplication in $\Q_{\pp}$ extend the respective operations
on those elements \eqref{eq:padic} which are finite sums (i.e., $a_k =0$ for all
but finitely many $k$), performed within the rational numbers. That is, carries 
to higher positions of coefficients that exceed $\pp-1$ have to be taken care
of.

We have the following for all $\uf, \vf \in \Q_{\pp}$ and $m \in \Z$.
\begin{equation}\label{eq:ultra}
\begin{aligned}
{\rm (i)} \qquad &|\uf|_{\pp} = 0 \iff \uf =0\,,\\ 
{\rm (ii)} \qquad &|\uf+\vf|_{\pp} \le \max\{ |\uf|_{\pp}, |\vf|_{\pp}\} \,,\\ 
{\rm (iii)} \qquad &|\uf\, \vf|_{\pp} \le |\uf|_{\pp}\, |\vf|_{\pp} \,,\\
{\rm (iv)} \qquad &|\pp^m\, \vf|_{\pp} = \pp^{-m} |\vf|_{\pp} \,.
\end{aligned}
\end{equation}
If $\pp$ is prime, then we always have equality in (iii), and $\Q_{\pp}$ is a 
field. Otherwise, it is only a ring. 
By (ii), the norm induces an ultrametric. Any metric ball in $\Q_{\pp}$ is open 
and compact, and $\Q_{\pp}$ is totally disconnected (a Cantor set). Let 
$\ol B(\uf, \pp^{-k})$ be the closed ball with radius $\pp^{-k}$ and centre 
$\uf$.
Each of its points is a centre for that ball.
It is the disjoint union of $\pp$ closed balls with radius $\pp^{-k-1}$.
Now consider
$$
H_m = \{ v = \ol B(\uf, \pp^{-m}) : \uf \in \Q_{\pp}\}\,.
$$
This is going to be the horocycle at level $m$ of our tree, and for
$v = \ol B(\uf, \pp^{-m})$ as a vertex of $\T=\T_{\pp}\,$, its predecessor is
$v^-= \ol B(\uf, \pp^{-m+1})$. This gives us the tree structure. We find that
$\bd^*\T = \Q_{\pp}\,$. 

We see that via the matrix representation \eqref{eq:matrix}, $\BS(\pp)$
acts on $\bd^*\T$ by affine transformations of the ring $\Q_{\pp}\,$.
This action extends to the tree: 
if $\gamma = \bigl(\begin{smallmatrix} \pp^n & k/\pp^l \\ 0 & 1 
\end{smallmatrix}\bigr)$
and $v = \ol B(\uf, \pp^{-m})$ then
$$
\gamma v = \bigl\{ \pp^n \vf + k/\pp^l : \vf \in \ol B(\uf, \pp^{-m}) \bigr\} 
= \ol B(\pp^n \uf + k/\pp^l, \pp^{-m-n})
$$
is a closed ball with radius $\pp^{-m-n}$, so that it is another vertex of our 
tree, which lies in $H_{m+n}\,$: (iv) of \eqref{eq:ultra} 
is crucial here. In this way, $\gamma$ defines an element of
$\Aff(\T_{\pp})\,$.

We can now take the diagonal embedding $\gamma \mapsto (\gamma,\gamma)$ of
$\BS(\pp)$ into $\Aff(\Hb,\pp) \times \Aff(\T_{\pp})$. This embedding is
compatible with the level structure of both $\Hb$ and the tree, so that
$\BS(\pp)$ is embedded into $\Af(\pp,\pp)$. It is easily seen to be 
discrete. It is co-compact because the factor space is compact. Indeed, a 
fundamental domain for the action of $\BS(\pp,\pp)$ is obtained as follows:
In $\Hb$, take the Euclidean (!) rectangle $R$ with vertices $\im$, $\pp+ \im$, 
$\pp + \im \pp$ and $\im\pp$. Then
$\{(z,w) \in \HT : z \in R\,,\; w \in [o^-,o]\}$ is a fundamental domain. 
The reader is invited to
elaborate these last details as an exercise; compare once more with \cite{FaMo}.  
\end{proof}

\begin{rmk}\label{rmk:density-tree} 
In the analogy between tree and hyperbolic upper half plane, $\partial^* \T$ corresponds
to the lower boundary line $\R$ of $\Hb$. In this spirit, the natural analogue of Lebesgue
measure on $\R$ is the measure $\la^*$ on $\partial \T$ which corresponds to (suitably 
normalised) Haar measure on the Abelian group $\Q_p$ under the identification of 
$\partial^*\T$ with $\Q_p\,$.
The basic open-closed sets in $\partial^*\T$ (the ultrametric balls) and their 
measures are 
$$
\partial_v^*\T = \{\xi \in \partial \T : v \in \geo{\varpi\,\xi} \} 
\AND \la^*(\partial_v^*\T) = \pp^{-\hor(v)}\,,\quad v \in V(\T).
$$
\end{rmk}

\section{Laplacians with drift}\label{sec:Laplacians}
We now explain our family of natural Laplace operators 
$\Lap^{\HT}=\Lap_{\al,\beta}^{\HT}$ on
$\HT(\qq,\pp)$ with ``vertical drift'' parameters $\al \in \R$ and $\beta > 0$.
Their rigorous construction is carried out in detail in \cite{BSSW}.
Here, we reproduce the basic facts.

\begin{dfn}\label{def:Cinfty} We let  $\CC^\infty(\HT)$ be the set of 
those continuous functions $f$ on $\HT$ such that, for each $v\in V(\T)$, the 
restriction $f_v$ of $f$ to the strip $\Sf_v$ (as in Definition
\ref{dfn:fv}) 
has continuous derivatives $\partial_x^m\partial_y^n f_v(z)$ of all orders 
in the interior $\Sf^o_v$ which satisfy, for all $R>0$,
$$
\sup\bigl\{|\partial_x^m\partial_y^n f_v(z)|:
z=x+\im y \in \Sf^o_{\hor(v)}\,,\; |\RE z| \le R\bigr\}
<\infty\,.
$$
\end{dfn}

Thus, on each strip $\Sf_v\,$, each partial derivative has a continuous
extension $\partial_x^m\partial_y^n f_v(z)$  
to the strip's boundary. Note that when $w^-=v$, it is in general \emph{not} true that 
$\partial_x^m\partial_y^n f_w = \partial_x^m\partial_y^n f_v$
on $\Lf_v = \Sf_v \cap \Sf_w\,$, unless $m=n=0$. We have the (hyperbolic) gradient
$\nabla f$ given by
$$
\nabla f_v(z) = \bigl(y^2 \partial_x f_v(z)\,,\,y^2 \partial_y f_v(z)\bigr)
$$ 
which is defined without ambiguity in the interior of each strip.
However, on any bifurcation line $\Lf_v\,$, we have to distinguish
between all the one-sided limits of the gradient, obtaining the family
$$
\nabla f_v(z) \AND \nabla f_w(z) \quad \text{for all}\; w \in V(\T)\;
\text{with}\; w^-=v\,,\quad (z,v) \in \Lf_v\,.  
$$ 

Let $\CC^\infty_c(\Omega)$
be the space of those functions in $\CC^\infty(\HT)$ that have compact 
support contained in $\Omega\,$.
We shall write
\begin{equation}\label{eq:LT}
\LT = \bigcup_{v \in V(\T)} \Lf_v \AND \HT^o =  \bigcup_{v \in V(\T)} \Sf_v^o = 
\HT \setminus \LT\,.
\end{equation}
For $\al\in \R\,$, $\beta>0$, we define the measure $\ms_{\al,\beta}$
on $\HT$ by
\begin{equation}\label{eq:mab}
\begin{gathered}
d\ms_{\alpha,\beta}(\zf)= \phi_{\al,\beta}(\zf)\,d\zf\quad\text{with}\\[3pt]
\phi_{\al,\beta}(\zf) = \beta^{\hor(v)}\,y^\alpha \quad \text{for}\quad
\zf=(x+\im y,w)\in \Sf_v \setminus \Lf_{v^-}\,,\quad \text{where }\;v \in V(\T)\,,
\end{gathered}
\end{equation}
that is, $w \in (v^-\,,v]$ and $\log_{\qq} y = \hor(w)$.

\begin{dfn}\label{def:ADHT}   For $f\in \CC^\infty(\HT)$
and $\zf=(x+\im y,w)\in \HT^o$, we set
$$
\Lap_{\al,\beta}f(\zf)= y^{2}(\partial_x^2+\partial_y^2)f(\zf)
+\alpha \,y\, \partial_yf(\zf)\,.
$$
Let $\DC^\infty_{\alpha,\beta,c}$ be the
space of all functions $f \in \CC^\infty_c(\HT)$ with the following
properties. 
\\[5pt]
(i) For any $k$,  the  $k$-th iterate $\Lap_{\al,\beta}^k f$, originally 
defined on $\HT^o,$  admits a continuous extension to all of $\HT$ (which then
belongs to $\CC^\infty_c(\HT)$ and is also denoted $\Lap_{\al,\beta}^k f$).
\\[5pt]
(ii) The function $f$, as well as each of its iterates $\Lap_{\al,\beta}^k f$,
satisfies the \emph{bifurcation conditions}
\begin{equation}\label{eq:bif}
\partial_y f_v=\beta\sum_{w\,:\,w^-=v}
\partial_y f_w \quad
\text{on $L_v$ for each $v \in V(\T)\,$.}
\end{equation}
\end{dfn}
$\Lap_{\al,\beta}$ as a differential operator on $\HT^o$ apparently depends only
on $\al$. Dependence on $\beta$ is through the domain of functions
on which the Laplacian acts, which have to satisfy \eqref{eq:bif}.
In the following propositions, we present some of the essential properties
proved in \cite{BSSW}, where additional details can be found.

\begin{pro}\label{pro:saHT}
The space $\mathcal D^\infty_{\alpha,\beta,c}$ is dense in the Hilbert space
${\mathcal L}^2(\HT,\ms_{\alpha,\beta})$. 

The operator $(\Lap_{\al,\beta},\mathcal D^\infty_{\alpha,\beta,c})$ 
is essentially self-adjoint in ${\mathcal L}^2(\HT,\ms_{\alpha,\beta})$. 
\end{pro}
With a small abuse of notation, we write 
$\bigl(\Lap_{\al,\beta},\Dom(\Lap_{\al,\beta})\bigr)$ 
for its unique self-adjoint extension. Indeed, at a higher level of rigour
in \cite[Def. 2.16]{BSSW}, the differential operator of Definition 
\ref{def:ADHT}(i), defined on $\mathcal D^\infty_{\alpha,\beta,c}\,$, is denoted 
${\mathfrak A}_{\al}\,$, and the notation $\Lap_{\al,\beta}$ is reserved for
the extension. The detailed construction of the latter in \cite{BSSW}
is carried out via Dirichlet form theory.

\begin{pro}\label{pro:sg}
\emph{(a)} The heat semigroup $e^{t\Delta_{\alpha,\beta}}$, $t>0$,
acting on $\LC^2(\HT,\ms_{\alpha,\beta})$ admits a continuous positive 
symmetric transition kernel
$
(0,\infty)\times \HT\times\HT\ni (t,\wf,\zf)\mapsto
\hb_{\alpha,\beta}(t,\wf,\zf)
$
such that for all $f\in \CC_c(\HT)$,
$$
e^{t\Delta_{\al,\beta}}f(\zf)=\int_{\HT}
\hb_{\alpha,\beta}(t,\wf,\zf)\,f(\zf)\,d\ms_{\al,\beta}(\zf)\,.
$$
\emph{(b)} For each fixed $(t,\wf)$, the function 
$\zf\mapsto \hb_{\al,\beta}(t,\wf,\zf)$  is
in $\CC^\infty(\HT)$ and satisfies \eqref{eq:bif}.
\\[5pt]
\emph{(c)} The heat semigroup is conservative, that is, 
$\int_{\HT} \hb_{\al,\beta}(t,\wf,\cdot)\,d\ms_{\al,\beta}=1$.
\\[5pt]
\emph{(d)} It sends $\LC^\infty(\HT)$ into $\CC^\infty(\HT)\cap  \LC^\infty(\HT)$
and $\CC_0(\HT)$, the space of continuous functions vanishing at infinity, into itself.
\end{pro}

The general theory of Markov processes tells us that $\Lap_{\al,\beta}$
is the infinitesimal generator of a Hunt process $(X_t)_{t \ge 0}\,$. This
is our Brownian motion on $\HT$.
It is defined for every starting point $\wf\in\HT$, has infinite life time 
and continuous sample paths.
Its family of distributions $(\mathbb P^{\alpha,\beta}_\wf )_{\wf\in \HT}$
on $\boldsymbol{\Omega}=\mathcal C ([0,\infty]\to \HT)$ is determined by
the one-dimensional distributions
$$
\mathbb P^{\alpha,\beta}_{\wf}[X_t\in U]
= \int_U \hb_{\al,\beta}(t,\wf,\zf) \,d\ms_{\al,\beta}(\zf)
= \int_U \pb_{\al,\beta}(t,\wf,\zf) \,d\zf\,,
$$
where $U$ is any Borel subset of $\HT$ and
$$
\pb_{\al,\beta}(t,\wf,\zf) = \hb_{\al,\beta}(t,\wf,\zf)\,\phi_{\al,\beta}(\zf)
$$
with the function $\phi_{\al,\beta}$ as in \eqref{eq:mab}. We note that this 
transition density with respect to $d\zf$ is invariant under the action of the
group $\Af$ of Theorem \ref{thm:isogroup}:
\begin{equation}\label{eq:groupinvariance}
\pb_{\al,\beta}(t,\gf\wf,\gf\zf) = \pb_{\al,\beta}(t,\wf,\zf) 
\quad\text{for all}\quad t > 0\,,\; \wf, \zf \in \HT \;\text{and}\; 
\gf \in \Af\,.
\end{equation}
We next say a few words about the natural projections of $\HT$.
We have 
$$
\begin{gathered}
\pi^{\Hb}: \HT \to \Hb\,,\;\ \zf = (z,w) \mapsto z\,,\quad
\pi^{\T}: \HT \to \T\,,\;\ \zf = (z,w) \mapsto w\,, \AND
\\
\pi^{\R}: \HT \to \R\,,\;\ \zf=(z,w) \mapsto \log_{\qq} \IM(z). 
\end{gathered}
$$
We also interpret $\pi^{\R}$ as a projection $\Hb \to \R$, where
$z \mapsto \log_{\qq} \IM(z)$, and as a projection $\T \to \R$,
where $w \mapsto \hor(w)$. Thus, the following diagram commutes.
$$
\beginpicture 

\setcoordinatesystem units <.54mm,.54mm>
\setplotarea x from -20 to 20, y from -20 to 23

\arrow <6pt> [.2,.67] from 0 17 to 0 -17
\arrow <6pt> [.2,.67] from 3 17 to 17 3
\arrow <6pt> [.2,.67] from -3 17 to -17 3
\arrow <6pt> [.2,.67] from -17 -3 to  -3 -17
\arrow <6pt> [.2,.67] from  17 -3 to 3 -17

\put {$\HT$} [c] at 0 20
\put {$\Hb$} [c] at -20 0 
\put {$\T$} [c] at 20 0
\put {$\R$} [c] at 0 -20

\put {$\pi^{\Hb}$} [rb] at -9 10
\put {$\pi^{\T}$} [lb] at 10.5 10 
\put {$\pi^{\R}$} [rt] at -10.5 -9.5
\put {$\pi^{\R}$} [lt] at 9 -9.5 
\put {$\pi^{\R}$} [lt] at 0.5 1

\endpicture
$$
The ``sliced'' hyperbolic plane as in Figure~3 can be interpreted as 
$\HT(\qq,1)$, that is, the tree is $\Z$, the bi-infinite line graph.
Everything that has been said above also applies here, so that
we have the operator $\Lap^{\Hb}_{\al,\beta}$ on $\Hb$. 

Analogously, we have a Laplacian $\Lap^{\T}_{\al,\beta}$ on the metric
tree, introduced in the same way as above. However, we should take care
of the slightly different parametrisation, that is, the stretching factor
$\log \qq$ in the construction of $\HT$, while in $\T$, each edge $[v^-,v]$
corresponds to the real interval $[\hor(v)-1\,,\hor(v)]$. The functions that
we consider now depend on one real variable in each open edge.
We write $f_v$ for the restriction of $f: \T \to \R$ to $[v^-,v]$.
We have to redefine the analogue of the measure of \eqref{eq:mab}:

\begin{equation}\label{eq:mabT}
\begin{gathered}
d\ms_{\alpha,\beta}^{\T}(w)= \phi_{\al,\beta}^{\T}(w)\,dw\quad\text{with}\\[3pt]
\phi_{\al,\beta}^{\T}(w) = \beta^{\hor(v)}\,\qq^{(\al-1)\hor(w)}\,\log \qq\quad \text{for}\quad
w \in (v^-,v]\,,
\end{gathered}
\end{equation}
where $v \in V(\T)$ and (recall) $dw$ is
the standard Lebesgue measure in each edge. The space $\CC^\infty(\T)$
is defined as in Definition \ref{def:Cinfty}, considering the edges of
$\T$ as the strips. 
The analogues of the crucial Definition \ref{def:ADHT} plus the bifurcation 
condition 
\eqref{eq:bif} now become the following:
Every $f \in \Dom(\Lap^{\T}_{\al,\beta})
\cap \CC^\infty(\T)$ must satisfy for every $v \in V(\T)$ 
\begin{equation}\label{eq:bif-LapT}
\begin{aligned} f_v'(v)&=\beta\sum_{w\,:\,  w^- = v} f_w'(v)\AND\\
\Lap^{\T}_{\al,\beta}f 
&=\frac{1}{(\log \qq)^2} f'' +  \frac{\alpha-1}{\log \qq} f'\quad 
\text{in the open edge}\;(v^-,v)\,.
\end{aligned}
\end{equation}
Finally, the analogue on the real line is comprised in the above by identifying
$\R$ with the tree with branching number $1$ (degree 2). In this case, the
vertices are the integers, the edges are the intervals $[k-1\,,\,k]$, where
$k \in \Z$, and the Laplacian becomes $\Lap^{\R}_{\al,\beta}\,$. Its definition
as a differential operator in each open interval $(k-1\,,\,k)$ is the same 
as in \eqref{eq:bif-LapT}, while the bifurcation condition becomes 
$f'(k-) = \beta\, f'(k+)$ for all $k \in \Z$.

With these modifications, propositions \ref{pro:saHT} and \ref{pro:sg}  
apply to all those Laplacians. 
We write $\hb_{\al,\beta}^{\Hb}\,$, $\hb_{\al,\beta}^{\T}$ and 
$\hb_{\al,\beta}^{\R}$ for the respective associated transition kernels.

\begin{pro}\label{pro:projections} Let $(X_t)$ be the process on $\HT(\qq,\pp)$ 
whose infinitesimal generator is $\Lap_{\al,\beta}\,.$
Set 
$$
Z_t=\pi^{\Hb}(X_t)\,,\quad W_t=\pi^{\T}(X_t)\,, \AND Y_t=\pi^{\R}(X_t)\,,\quad 
t \ge 0.
$$
\emph{(a)}  The process $(Z_t)$ is a Markov process on $\Hb$ whose 
infinitesimal generator is $\Lap^{\Hb}_{\al,\beta\pp}$. Its
transition kernel with respect to the measure $\ms_{\al,\beta\pp}^{\Hb}$
is $\hb_{\al,\beta\pp}^{\Hb}\,$.
\\[4pt]
\emph{(b)}  The process $(W_t)$ is a Markov process on $\T$ whose 
infinitesimal generator is $\Lap^{\T}_{\al,\beta}$. Its
transition kernel with respect to the measure $\ms_{\al,\beta}^{\T}$
is $\hb_{\al,\beta}^{\T}\,$.
\\[4pt]
\emph{(c)}  The process $(Y_t)$ is a Markov process on $\R$ whose 
infinitesimal generator is $\Lap^{\R}_{\al,\beta\pp}$. Its
transition kernel with respect to the measure $\ms_{\al,\beta\pp}^{\R}$
is $\hb_{\al,\beta\pp}^{\R}\,$.
\end{pro}
 
\begin{dfn}\label{dfn:exit}
For any open domain $\Omega \subset \HT$, we let 
$\tau^{\Omega} = \inf \{ t>0: X_t \in \HT \setminus \Omega\}$ be the 
\emph{first exit time} of $(X_t)$ from $\Omega$,
and if $\tau = \tau^{\Omega} < \infty$ almost surely for the starting point
$X_0 = \wf \in \Omega$, then we write $\mu_{\wf}^{\Omega}$ for the
distribution of $X_{\tau}\,$. \end{dfn}

$\mu_{\wf}^{\Omega}$ is a probability measure which usually
is supported by $\partial \Omega$ (we do not specify the meaning of ``usually'';
for the sets that we are going to consider, this will be true).
We shall use analogous notation on $\Hb$, $\T$ and $\R$. We note that
\begin{equation}\label{eq:inv-meas}
\mu_{\wf}^{\Omega}(B) = \mu_{\gf\wf}^{\gf\Omega}(\gf B) \quad\text{for every}\; 
\gf \in \Af \;\text{and Borel set}\; B \subset \HT \,.
\end{equation}

\begin{dfn}\label{def:harmonic} Let $\Omega \subset \HT$ be open.
A continuous function $f:\Omega  \to \R$ is called \emph{harmonic on $\Omega$}
if for every open, relatively compact set $U$ with $\ol U \subset \Omega$,
$$
f(\zf) = \int f\, d\mu_{\zf}^U \quad \text{for all}\; \zf \in U\,.
$$
\end{dfn}

From the classical analytic viewpoint, this defintion may be unsatisfactory;
``harmonic'' should mean ``annihilated by the Laplacian'' (as a differential operator). 
However, for general open domains in $\HT$, the correct formulation in these terms is 
quite subtle in view of the relative location of the bifurcations. 
 
More details will be stated and used in \cite{BSSW3}.

\section{Brownian motion and the induced random walks}\label{sec:rw}
 
Our basic approach is to study BM on $\HT$ via the random walk resulting from
observing the processes during its successive visits in the set $\LT$ of all 
bifurcation lines.

Thus, we define the stopping times $\tau(n)$, $n \in \N_0\,$, 
\begin{equation}\label{eq:stop}
\tau(0)=0, \quad \tau(n+1) 
= \inf\bigl\{ t > \tau(n) : Y_t \in \Z \setminus \{ Y_{\tau(n)} \}\bigr\}\,.
\end{equation}
They are not only the times of the successive visits of $(Y_t)$ in $\Z$:
by Proposition \ref{pro:projections}, if $X_0$ lies in some open strip $\Sf_v^o$, 
then $\tau(1)$ is the
exit time from that open strip, that is, the instant when $X_t$ first meets a point
on $\Lf_v \cup \Lf_{v^-}\,$.  If $X_{\tau(n)} \in \Lf_v$ for some $v \in V(\T)$
(which holds for all $n \ge 1$, and possibly also for $n=0$), then
$\tau(n+1)$ is the first instant $t > \tau(n)$ when $X_t$ meets one of the
bifurcation lines $\Lf_{v^-}$ or $\Lf_w$ with $w^- = v$.
The $\tau(n)$ are also  the times of 
the successive visits of $(Z_t)$ in the union of all the lines $\Lf_k$
that subdivide $\Hb$, as well as the times of the successive visits of $(W_t)$ in 
the vertex set $V(\T)$ of $\T$. 
Later on, we shall also need the integer random variables $\nb_t$, defined by
\begin{equation}\label{eq:nt}
\tau(\nb_t) \le t < \tau(\nb_t+1)\,,\quad \text{where}\; t \ge 0\,,
\end{equation}
as well as the stopping time
\begin{equation}\label{eq:sig}
\sigma = \inf\bigl\{ t > 0 : Y_t \in \{ -1, 1\}\bigr\}\,, \quad
\text{where} \quad Y_0 =y_0 \in [-1\,,\,1].
\end{equation}
This is the exit time from $[-1\,,\,1]$.
Note that $\sigma = \tau(1)$ when  $y_0 = 0$, but not when $0 < |y_0| < 1$.

\begin{lem}\label{lem:tau-finite} For any starting point in $\R$, 
resp. $[-1\,,\,1]$,
the stopping times $\tau(1)$ and $\sigma$ are almost surely finite.
\end{lem}

\begin{proof}
We start with $\sigma$. Consider the function $g(y) = \Prob_y[\sigma < \infty]$
on $[-1\,,1]$. It is a weak solution of the Dirichlet problem 
$\Lap^{\R}_{\al,\beta\pp}g = 0$ on $(-1\,,\,1)$ with boundary values 
$1$ at $\pm 1$. By \cite[Theorem 5.9]{BSSW} (in a simplified version, because
here we are dealing with the infinite line as a metric graph), $g$ is a 
strong solution. Thus, $g$
satisfies the following ``broken'' differential equation, where we
have to use $\Lap^{\R}_{\al,\beta\pp}$.
\begin{equation}\label{eq:broken}
\frac{1}{(\log \qq)^2} g'' +  \frac{\alpha-1}{\log \qq} g' = 0\,,\qquad
g'(0-) =\beta\pp \, g'(0+)\,,
\end{equation}
with $g(\pm1) =1$.
The unique solution is $g \equiv 1$, whence 
$\Prob_y[\sigma< \infty] = g(y) = 1$.

Now let us consider $\tau(1)$. By \eqref{eq:groupinvariance}, the transition 
density of  $(Y_t)$ is invariant under translation by integers. Therefore we 
may suppose that the starting point is in 
$[0\,,\,1)$. If it is $0$ then $\tau(1)=\sigma$, so we restrict to starting points
$y \in  (0\,,\,1)$. Set $h(y) = \Prob_y[\tau(1) < \infty]$.
Then $h$ satisfies the differential equation
$$
\frac{1}{(\log \qq)^2} h'' +  \frac{\alpha-1}{\log \qq} h' = 0
$$
on the interval $(0\,,\,1)$, with boundary values $h(0) = h(1) = 1$.
Again, the unique solution is $h \equiv 1$, whence $\Prob_y[\tau(1)< \infty] =1$.
\end{proof}

We shall need detailed computations regarding the two integer random variables
$$
\tau = \tau(2) - \tau(1) \AND Y = Y_{\tau(2)} - Y_{\tau(1)},
$$
in particular their expected values and variances. Note that
$Y$ takes the values $\pm 1$.

\begin{pro}\label{pro:tau} 
\emph{(a)} The increments $\tau(n)-\tau(n-1)$,
$n \ge 1$, are independent and almost surely finite.
\\[4pt]
\emph{(b)} They are identically distributed for $n\ge 2$, and when $Z_0 \in \LT$,
then also $\tau(1)$ has the same distribution.\\[4pt]
\emph{(c)}
Let $\;  b = \dfrac{(\alpha-1)\log \qq}{2} \;$ and consider the real functions
$\; s(\la) =  b^2 + (\log \qq)^2\,\la\;$ and 
$$
r(\la) = (\beta\,\pp +1) \sum_{n=0}^{\infty} \frac{s(\la)^n}{(2n)!} + 
(\beta\,\pp - 1)\,b\, \sum_{n=0}^{\infty} \frac{s(\la)^n}{(2n+1)!}\,, \quad \la \in \R\,.
$$
Then the random variables $Y$ and $\tau$ defined above are independent,
$$
\Ex(e^{-\la \tau}\uno_{[Y=1]}) =\beta\, \pp \,e^b/r(\la) \AND
\Ex(e^{-\la \tau}\uno_{[Y=-1]}) = e^{-b}/r(\la). 
$$
\emph{(d)} 
In particular, setting $\rha = \beta \pp \,\qq^{\alpha-1}$.
$$
\begin{aligned}
&\Prob[Y = 1] = \frac{\rha}{\rha+1}\,,\quad 
\Prob[Y = -1] = \frac{1}{\rha+1}\,,\\[3pt]
&\Ex(Y) = \frac{\rha-1}{\rha+1}\,,\AND 
\Var(Y) = \frac{4\rha}{(\rha+1)^2}\,.
\end{aligned}
$$
\emph{(e)} The Laplace transform $\la \mapsto \Ex(e^{-\la \tau})
= (\rha+1)e^{-b}\big/r(\la)$
is analytic in a neighbourhood of $0$, so that $\tau$ has finite exponential moment
$\Ex(e^{\la_0 \tau})$ for some $\la_0 > 0$.
\\[4pt]
\emph{(f)} The expectation and variance of $\tau$ are
$$
\Ex(\tau) = r'(0)e^b/(\rha+1) \AND \Var(\tau) = \Ex(\tau)^2 - r''(0)e^b/(\rha+1).
$$ 
\end{pro}
 
\begin{proof}
(a) and (b) are clear. 

For (c), we fix $\la \ge 0$ and consider again the exit time $\sigma$ 
of the process
$(Y_t)$ from the interval $[-1\,,\,1]$. We let
$f_{\pm 1}(y) = \Ex_y(e^{-\la \sigma} \uno_{[Y_{\sigma}=\pm 1]})$, respectively,
defined for the 
starting point $Y_0= y \in [-1\,,\,1]$. We note that
$\Ex(e^{-\la \tau}\uno_{[Y=\pm 1]}) = f_{\pm 1}(0)$. Each of the two functions 
$f_{\pm 1}$  is a weak, whence 
strong \cite[Theorem 5.9]{BSSW} solution of the Dirichlet problem 
$\Lap^{\R}_{\al,\beta\pp}f_{\pm 1} = \la \cdot f_{\pm 1}$ on the interval 
$[-1\,,\,1]$
with boundary values $0$ and $1$, resp. $1$ and $0$ at the endpoints $-1$ and $1$.
Thus, $f_{-1}$ and $f_1$
satisfy the ``broken'' differential equation
$$
\begin{gathered}
\frac{1}{(\log \qq)^2} f_{\pm 1}'' +  \frac{\alpha-1}{\log \qq} f_{\pm1}' = 
\la \cdot f_{\pm 1}\,,\qquad
f_{\pm 1}'(0-) =\beta\pp \, f_{\pm 1}'(0+)\,,\\[4pt]
f_1(-1) = 0\,, f_1(1) = 1\,, \quad \text{resp.} \quad f_{-1}(-1) =1\,,
f_{-1}(1) =0.
\end{gathered}
$$
The computation of the solutions is a lengthy, but basic exercise that leads to
(c); it may be useful here to note that
$$
r(\la) = \begin{cases}
          (\beta\,\pp +1) \cosh \sqrt{s(\la)} \;+\; 
          (\beta\,\pp - 1)\,b\,\,\sinh\! \sqrt{s(\la)}\Big/\!\sqrt{s(\la)}\,,\; & s(\la) \ge 0\,,\\[6pt]
          (\beta\,\pp +1) \cos \sqrt{-s(\la)} \;+\; 
          (\beta\,\pp - 1)\,b\,\sin\! \sqrt{-s(\la)}\Big/\!\sqrt{-s(\la)}\,,\; & s(\la) \le 0\,.
         \end{cases}
$$ 

Statement (d) is obtained by setting $\la=0$ in (c).

A short computation 
now shows that 
$\Ex(e^{-\la \tau}\uno_{[Y=\pm 1]}) =\Ex(e^{-\la \tau})\Prob([Y=\pm 1]$
for all $\la \ge 0$, which yields independence of $Y$ and $\tau$.

Statement (e)  is obvious from the form of the Laplace transform.

Statement (f) is obtained by
direct computations of the first and second derivatives of the transform.
\end{proof}
We can compute 
\begin{equation}\label{eq:Etau}
\Ex(\tau) =  \begin{cases}
\dfrac{(\log \qq)^2}{2b^2}\,  
\dfrac{(\beta\pp\!-\!1) b \cosh b + [(\beta\pp\!+\!1)b - (\beta\pp\!-\!1)] \sinh b}
{(\beta\pp\!+\!1) \cosh b + (\beta\pp\!-\!1) \sinh b}\,,&\text{if}\; \alpha \ne 1\,,
\\[4pt]
\dfrac{(\log \qq)^2}{2}\,,&\text{if}\; \alpha = 1\,, 
\end{cases}
\end{equation}
However, we omit the lengthy formula for $\Var(\tau)$, wich can be obtained by 
tedious computation but provides no specific insight. 

The following is obtained by completely similar, but simpler computations. 
(Namely, we have to solve the same differential equation as above for computing 
$\Ex_y(e^{-\la \tau(1)})$, but it is not ``broken''.)

\begin{lem}\label{lem:expmom}
For any $y \in \R$, there is $\la = \la(y) > 0$ (depending only on the fractional
part of~$y$) such that for the process $(Y_t)$ starting at $y$
$$
\Ex_y\bigl(e^{\la \tau(1)}\bigr) < \infty\,.
$$
\end{lem}

We now clarify the nature of the induced processes on $\Z$ and on $\T$, 
respectively.

\begin{cor}\label{cor:RW} With  $\rha = \beta \pp \,\qq^{\alpha-1}$
as in Proposition \ref{pro:tau}(d),
\\[4pt]
\emph{(a)} the process 
$\bigl( Y_{\tau(n)} \bigr)_{n \ge 1}$ is a nearest neighbour
random walk on $\Z$ with transition probabilities
$$
p_{\Z}(k,l) = \Prob[Y_{\tau(n+1)} = l\mid Y_{\tau(n)} = k] = 
\begin{cases} \dfrac{\rha}{1+\rha}\,,&\text{if }\; l = k+1\,,\\[10pt]
\dfrac{1}{1+\rha}\,,&\text{if }\; l = k-1\,,\\[6pt]
0\,,&\text{otherwise.}
\end{cases}
$$
\emph{(b)} The process 
$\bigl( W_{\tau(n)} \bigr)_{n \ge 1}$ is a transient nearest neighbour
random walk on (the vertex set of) $\T$ with
transition probabilities
$$
p_{\T}(v,w) = \Prob[W_{\tau(n+1)} = w\mid W_{\tau(n)} =v] = 
\begin{cases} \dfrac{\rha}{(1+\rha)\pp}\,,&\text{if }\; w^-=v\,,\\[10pt]
\dfrac{1}{1+\rha}\,,&\text{if }\; w = v^-\,,\\[6pt]
0\,,&\text{otherwise,}
\end{cases}
$$
where $v,w \in V(\T)$.
\end{cor}

\begin{proof} (a) is immediate from Proposition \ref{pro:tau}(d).
\\[4pt]
(b) is an immediate consequence of (a), because
for any $v \in V(\T)$, we must have $p_{\T}(v,v^-) = p_{\Z}(k,k-1)$, while
$p_{\T}(v,w)$ must be the same for all successors $w$ of $v$, with sum
$p_{\Z}(k,k+1)$. It is well-known and easy to prove that this random walk on 
$V(\T)$ is transient (visits any finite set only finitely often a.s.), compare with \cite{CKW} or 
\cite{Wlamp}.
\end{proof}

The transition kernel of the induced processes on $\HT$, resp. $\Hb$, cannot 
be computed as explicitly. 
We need to consider the non-compact set
\begin{equation}\label{eq:Omv}
\Omega_v = \{ (z,w) \in \HT : w \in N(v)^o\} \subset \HT\,,
\end{equation}
where $N(v)$ is the ``neighbourhood star'' in $\T$ at $v \in V(\T)$.
That is, $N(v)$ is the union of all edges ($\equiv$ intervals !) of $\T$ which
have $v$ as one endpoint. It is a compact metric subtree of $\T$, whose boundary
$\partial N(v)$ consists of all neighbours of $v$ in $V(\T)$.
 We write
$\partial^+ N(v) = \partial N(v) \setminus \{ v^-\}$ (the forward neighbours
of $v$).

For any starting point $\wf \in \Omega_v\,$, the exit time $\tau^{\Omega_v}$ 
is almost surely finite by \eqref{eq:groupinvariance} and Lemma 
\ref{lem:tau-finite}.
Thus, we have the probability measure $\mu_{\wf}^{\Omega_v}\,$ on the boundary of
$\Omega_v$ in $\HT$, 
$$
\partial \Omega_v = \bigcup_{w \in \partial N(v)} \Lf_w\,.
$$
For $\wf \in \Lf_v\,$, this is the transition probability of the Markov
chain $(X_{\tau(n)})$ on $\LT\,$: for any $\wf \in \Lf_v$ ($v \in V(\T)$)  
and Borel set $B \subset \partial \Omega_v\,$,
\begin{equation}\label{eq:transprob}
\Prob[X_{\tau(n+1)} \in B \mid X_{\tau(n)} = \wf] =  
\mu_{\wf}^{\Omega_v}(B)\,.
\end{equation}

\begin{lem}\label{lem:measure}  For any $\wf \in \Omega_v\,$,
the measure $\mu_{\wf}^{\Omega_v}$ is supported by whole of the boundary of
$\Omega_v$ in $\HT$, 
$$
\partial \Omega_v = \Lf_{v^-} \cup \bigcup_{w \in V(\T): w^- = v} \Lf_w\,.
$$
In particular, the process $(X_{\tau(n)})$ is irreducible on $\LT\,$: 
for any starting point $\wf \in \LT$ and any non-empty open interval $I$ 
that lies on one of the bifurcation lines,
$$
\Prob_{\wf}[\exists n : X_{\tau(n)} \in I] > 0\,.
$$ 
\end{lem}

\begin{proof} The second statement follows from the first one.
The first one follows from ellipticity of $\Lap_{\alpha,\beta}\,$.
More specifically, we can also see this as follows. A boundary point
$\zf$ of any open domain $\Omega \subset \HT$ is regular for the Dirichlet
problem with respect to $\partial \Omega$ if and only if 
$\Prob_{\zf}[\tau^{\Omega} = 0] =1$ (a general fact from Potential
Theory). 

Every boundary point of $\Omega_v$ \emph{is} regular.
This follows from the fact that $\tau^{\Omega_v}$ is the same as the exit time
of the process $(W_t)$ on $\T$ from the neghbourhood star $N(v)^o$.
But the Dirichlet problem for the latter (with boundary values at the finitely
many neighbours of $v$ in $V(\T)$) is obviously solvable, as one can verify by
direct, elementary computations similar to those used in the proof of 
Lemma \ref{lem:tau-finite}.

To conclude, recall that every regular point has to be in the support of 
the first exit measure. 
\end{proof}

We choose the point 
$\of = (\im,o) \in \HT$ as the \emph{origin} of treebolic space. Let 
$$
\mu = \mu_{\of}^{\Omega}\,,\quad \text{where} \quad \Omega = \Omega_o\,.
$$
By group invariance \eqref{eq:inv-meas}, we have
\begin{equation}\label{eq:degmu}
\mu_{\wf}^{\Omega_v} = \de_{\gf} * \mu\,, \quad \text{when} \quad
\gf \in \Af\,, \;\;\gf\of = \wf \in \Lf_v\,.
\end{equation}
The convolution of the Dirac measure at $\gf$ with $\mu$ is defined via the 
action of the group $\Af$, which is transitive on $\LT$. That is, $\LT$ is a 
homogeneous space of $\Af$ (the stabilizer of $\of$ in $\Af$ is a non-trivial
compact subgroup), and $(X_{\tau(n)})$ is a random walk on that
homogeneous space.

\smallskip

The transition kernel of  $(Z_{\tau(n)})$ can be obtained analogously.
That process evolves  on 
$$
\LH = \bigcup_{k\in \Z} \Lf_k \subset \Hb.
$$
We let
$$ 
\wt \Omega_k = \bigl(\Sf_{k-1} \cup \Sf_k\bigr)^o = \pi^{\Hb} (\Omega_v) 
\quad \text{for any}\; v \in H_k \subset V(\T)\,.
$$
Its boundary within $\Hb$ is $\partial\wt \Omega_k= \Lf_{k-1} \cup \Lf_{k+1}\,$.
For any starting point $z \in \wt \Omega_k$, we let
$\wt\mu_z^{\wt \Omega_k}$ be the exit distribution from  $\wt\Omega_k\,.$
In analogy with Lemma \ref{lem:measure}, it is supported by the whole
of $\partial\wt \Omega_k\,$, and
for any $z \in \Lf_k$   
and Borel set $B \subset \partial \wt \Omega_k\,$,
\begin{equation}\label{eq:transprob-H}
\Prob[Z_{\tau(n+1)} \in B \mid Z_{\tau(n)} = z] =  
\wt\mu_z^{\wt \Omega_k}(B)\,.
\end{equation}
We set
\begin{equation}\label{eq:wtmu}
\wt\mu = \wt\mu_{\,\im}^{\wt\Omega}\,,\quad{where} \quad \wt \Omega = \wt\Omega_0\,. 
\end{equation}
This is the image of $\mu$ under the projection $\pi^{\Hb}$.
Once more by group invariance \eqref{eq:inv-meas}, we have
$$
\wt\mu_{z}^{\wt\Omega_k} = \de_{g} * \wt\mu\,, \quad \text{when} \quad
g \in \Aff(\Hb,\qq)\,, \;\;g\im = z \in \Lf_k\,.
$$
Now we note that the group $\Aff(\Hb,\qq)$ acts \emph{simply transitively}
on $\LH$. Indeed, $\LH$  can be identified with $\Aff(\Hb,\qq)$
via the homeomorphic one-to-one correspondence
\begin{equation}\label{eq:identif} 
z= x +\im \qq^k \leftrightarrow g = \Bigl(\!\begin{smallmatrix}\qq^k &x\\[2pt] 0 & 1\end{smallmatrix}\!\Bigr)\,,\AND g\im = z\,.
\end{equation}
Thus, group invariance tells us that we can consider the process
$(Z_{\tau(n)})$ as the right random walk on $\Aff(\Hb,\qq)$ with law $\wt\mu$.
In other words, the increments  $Z_{\tau(n-1)}^{-1}Z_{\tau(n)}$, $n \ge 2$
(resp. $n \ge 1$, when $Z_0 \in \HT$) are i.i.d. random variables with 
distribution $\wt\mu$, when we consider inverses in $\Aff(\Hb,\qq)$
via the identification \eqref{eq:identif}.  
\begin{cor}\label{cor:transience-Ztaun}
The random walk $(Z_{\tau(n)})$ on $\Aff(\Hb,\qq)$ is transient.
\end{cor}

\begin{proof} The support of the probability 
measure $\wt\mu$ on $\Aff(\Hb,\qq)$ generates that group as a semigroup, 
that is, the random walk is irreducible
(every open set is reached with positive probability). We know from
\eqref{eq:modularAffRq} that the group $\Aff(\Hb,\qq)$ is non-unimodular.
Now, any irreducible random walk on a non-unimodular group must be transient,
see \cite{GKR}, or, for a shorter proof, \cite {Wrec}. 
\end{proof}

The remainder of this section is dedicated to a study of properties of
the probability measures $\mu$ on 
$\partial \Omega = \Lf_{o^-} \cup\bigcup_{v:v^- = o } \Lf_v \subset \LT$
and $\wt \mu$ on $\partial \wt \Omega = \Lf_{-1} \cup \Lf_1 \subset \LH$,
respectively. 

An important step is to show that in between two successive times $\tau(n)$ and
$\tau(n+1)$, the processes $(Z_t)$ on $\Hb$ and thus also $(X_t)$ on $\HT$ cannot escape
too far ``sideways'' within the the current strip (i.e., the strip to which
the process is confined between those two times).  

\begin{pro} \label{pro:sideways} Suppose $(Z_t)$ starts in  
$\wt \Omega$.
There is $\rho < 1$ such that for every $n \in \N$, we have for the exit time
$\sigma$ from $\wt\Omega$ 
$$
\Prob_{z_0} \Bigl[ \max \bigl\{ |\RE Z_t-\RE z_0| : 0 \le t \le \sigma \bigr\} \ge n \Bigr] 
\le 2\rho^n\,\quad\text{for every }\; z_0 \in \wt \Omega.
$$
(Recall again that $\sigma = \tau(1)$ when $z_0 \in \Lf_0$.)
\end{pro}

\begin{proof}
By invariance under horizontal translations, we may assume that $\RE z_0 =0$.

Consider the vertical segments, resp. open sets 
$$
J_n = \{ n + \im y : \qq^{-1} < y < \qq \} \subset \wt \Omega
\AND
\wt \Omega^{(n)} = \{ z \in \wt \Omega : \RE z < n \} \,,\quad n \in \Z\,,
$$
so that $J_n$ is the right hand side boundary of $\wt \Omega^{(n)}$.
For any starting point in $(\Sf_0 \cup \Sf_1)^o$, the exit time of $(Z_t)$ 
from $\wt \Omega$ is the $\sigma$ from \eqref{eq:sig}, and when 
$Z_0 \in \Lf_0$ then $\sigma =\tau(1)$. 
Analogously, we let $\sigma(n)$ be the exit time
of $(Z_t)$ from $\wt \Omega^{(n)}$. 

Now our argument will be as follows:
if $(Z_t)$ starts at $z_0$ and there is some $t \le \tau(1)$ such that 
$\RE Z_t \ge n$   then $(Z_t)_{t < \sigma}$ must pass through
each $J_k\,$, $k= 1, \dots, n$.  
$$
\beginpicture 

\setcoordinatesystem units <1.12mm,0.8mm>  

\setplotarea x from -40 to 40, y from 8 to 46

\arrow <6pt> [.2,.67] from 0 0 to 0 44

\plot -40 8  40 8 /
\plot -40 32  40 32 /

\plot -30 8  -30 32 /
\plot -20 8  -20 32 /
\plot -10 8  -10 32 /
\plot 30 8  30 32 /
\plot 20 8  20 32 /
\plot 10 8  10 32 /

\put {$\im$} [rb] at -0.3 16.6
\put {$\scs \bullet$} at 0 16
\put {$\Lf_{-1}$} [r] at -43 8
\put {$\Lf_0$} [r] at -43 16
\put {$\Lf_1$} [r] at -42 32

\put {$\Sf_0$} at 40 12
\put {$\Sf_1$} at 40 24

\put {$J_{-3}$} [b] at -30 33
\put {$J_{-2}$} [b] at -20 33
\put {$J_{-1}$} [b] at -10 33
\put {$J_3$} [b] at 30 33
\put {$J_2$} [b] at 20 33
\put {$J_1$} [b] at 10 33

\setdots <3pt>\plot -40 16  40 16 /

\endpicture
$$
\vspace{-.2cm}

\begin{center}
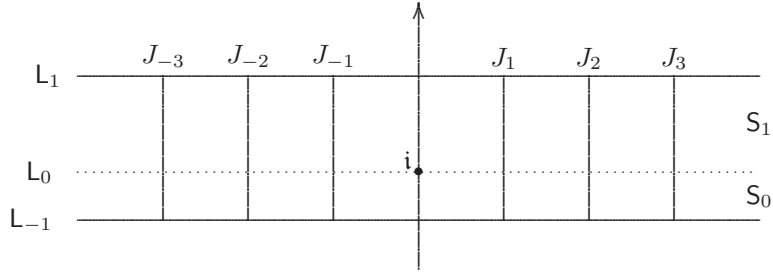
{\sl Figure~6.} The set $\wt \Omega \subset \Hb$, subdivided by
the bifurcation line $\Lf_0$\\ and the vertical segments $J_k\,$.
\end{center}

\smallskip

The function $z \mapsto \Prob_z[Z_{\sigma(1)} \in \Lf_{-1}\cup \Lf_1 ]$
is weakly harmonic (harmonic in the sense of distributions) on $\wt \Omega^{(1)}$, 
whence strongly harmonic by 
\cite[Theorem 5.9]{BSSW}, and thus continuous. We consider this function
on $\bar J_0$. At the endpoint of that segment, it is $=1$, while inside
$J_0$ it is $<1$. Thus, there is $z_0 \in J_0$ where our function attains its
minimum, and
$$
\rho =  1 - \Prob_{z_0}[Z_{\sigma(1)} \in \Lf_{-1}\cup \Lf_1] < 1\,.
$$
But then
$$
\Prob_z[Z_{\sigma(1)} \in J_1] 
= 1 - \Prob_z[Z_{\sigma(1)} \in \Lf_{-1}\cup \Lf_1] \le \rho \quad
\text{for every} \; z\in J_0\,.
$$
By invariance under the group $\Aff(\Hb,\qq)$, and in particular under translations
by reals, we also have for all $k \ge 1$
$$
\Prob_z[Z_{\sigma(k)} \in J_k] 
\le \rho \quad
\text{for every} \; z\in J_{k-1}\,.
$$
We now use ``balayage'' in probabilistic terms. Just for the next lines,
consider the measure $\wt \nu(B) = \Prob_{\im}[Z_{\sigma(k-1)} \in B]$ for Borel sets 
$B  \subset J_{k-1}\,$.   
If $Z_0 =\im$ and
$Z_{\sigma(k)} \in J_k$ then we must have $Z_{\sigma(k-1)} \in J_{k-1}\,$.
Therefore (by the strong Markov property)
$$
\begin{aligned}
\Prob_{z_0}[Z_{\sigma(k)} \in J_k] 
&=  \Prob_{z_0}[Z_{\sigma(k)} \in J_k\,,\; Z_{\sigma(k-1)} \in J_{k-1}]
= \int_{J_{k-1}}\!\! \Prob_z[Z_{\sigma(k)} \in J_k] \, d\wt\nu(z)\\
&\le \rho \cdot \wt\nu(J_{k-1}) 
= \rho \cdot
\Prob_{z_0}[Z_{\sigma(k-1)} \in J_{k-1}]\,.
\end{aligned}
$$
Inductively,
$$
\Prob_{z_0}[Z_{\sigma(k)} \in J_k] \le \rho^k \quad \text{for every}\; k \ge 1\,.
$$
If $Z_0=z_0$ and $\RE Z_t \ge n$ for some $t \le \sigma$  then a visit to
$J_n$ must have occured before time $t$. That is, $\sigma(n) \le t$, whence
\begin{equation}\label{eq:occur}
\Prob_{z_0}\bigl[ \max \{ \RE Z_t : 0 \le t \le \sigma \} \ge n \bigr] 
\le \rho^n\,. 
\end{equation}
Now observe that our process is also invariant under the reflection
$x + \im y \mapsto -x + \im y$. Therefore
$$
\Prob_{\im}\bigl[ \min \{ \RE Z_t : 0 \le t \le \sigma \} \le -n \bigr] 
\le \rho^n\,. 
$$
The proposed inequality follows.
\end{proof}

Relying again on group invariance \eqref{eq:groupinvariance}, 
we deduce the following.

\begin{cor}\label{cor:inter} The random variables
$$
\begin{aligned}
D_n &= \max \bigl\{ \dist_{\HT}(X_t, X_{\tau(n)}) : \tau(n) \le t \le \tau(n+1)\}\\
& = \max \bigl\{ \dist_{\Hb}(Z_t, Z_{\tau(n)}) : \tau(n) \le t \le \tau(n+1)\}
\,,
\quad n \in \N\,,
\end{aligned}
$$ 
are i.i.d. and 
$$
\limsup_{n \to \infty} \frac{D_n}{\log\log n} \le 2 \quad \text{almost surely.}  
$$
In particular,
$$
\Ex\Bigl(\exp \bigl(\exp(D_n/3)\bigr)\Bigr) < \infty\,.
$$
\end{cor}

\begin{proof}
It is clear that the $D_n$ are i.i.d. 
For the purpose of the proofs of this and the next corollary, set 
$$
M_n = \max \bigl\{ |\RE Z_t- \RE Z_{\tau(n-1)}| : \tau(n-1) \le t \le \tau(n) \bigr\}.
$$
These random variables are also i.i.d. With $\rho$ as in Proposition 
\ref{pro:sideways}, and for arbitrary $\ep > 0$,
$$
\sum_{n=2}^{\infty}
\Prob_{\im}[M_n \ge \tfrac{1+\ep}{\log(1/\rho)}\log n] \le
2 \sum_{n=2}^{\infty} \exp{\Bigl((\log\rho)\bigl\lfloor 
            \tfrac{1+\ep}{\log(1/\rho)}\log n\bigr\rfloor\Bigr)} 
\le \frac{2}{\rho} \sum_{n=2}^{\infty} \frac1{n^{1+\ep}}\,,
$$
which is finite. By the Borel-Cantelli Lemma,
$$
\limsup_{n\to \infty} \frac{M_n}{\log n} \le \frac{1}{\log(1/\rho)}
\quad\text{almost surely.}
$$
We also see that
\begin{equation}\label{eq:la1}
\Ex(e^{\la_1\,M_1}) < \infty \quad \text{for}\quad  0 < \la_1 < \log(1/\rho).
\end{equation}
By simple computations with the hyperbolic metric, for any 
$\zf = (z,w) \in \partial \Omega$, and thus
$z \in \partial \wt \Omega =\Lf_1 \cup \Lf_{-1}\,$, one has
\begin{equation}\label{eq:strip-dist}
\log(1 +|\RE z|^2) - \log\qq \le  d_{\HT}(\zf,\of) =
d_{\Hb}(z,\im) \le \log \qq +  2\log(1 + |\RE z|).
\end{equation}
Therefore $D_n \le \log \qq + 2\log (1+M_n)$, whence
as above,
$$
\sum_{n = 3}^{\infty} \Prob_{\im}[D_n \ge (2+\ep) \log\log n] < \infty
$$
for every $\ep > 0$. We get $\limsup D_n/\log\log n \le 2$ a.s.
Also, for some $c > 0$,
$e^{D_1/3} \le \qq^{1/3} (1+M_1)^{2/3} \le c + \la_1 M_1$.
Now \eqref{eq:la1} yields the doubly exponential moment condition for $D_1\,$. 
\end{proof}

From the last corollary and \eqref{eq:strip-dist}, we also get the
following.

\begin{cor}\label{cor:mumom}
With $\la_1 > 0$ as in \eqref{eq:la1},
$$
\int_{\partial \wt \Omega} \exp \bigl(\la_1 \RE z) \bigr)\, d\wt\mu(z) < \infty
$$
In particular, $\mu$ and $\wt \mu$ satisfy doubly exponential moment 
conditions: 
$$
\int_{\partial \wt \Omega} \exp \Bigl(\exp\bigl(d_{\Hb}(\im,z)/3\bigr)\Bigr)\, 
d\wt\mu(z) =
\int_{\partial \Omega} \exp \Bigl(\exp\bigl(d_{\HT}(\of,\zf)/3\bigr)\Bigr)\, 
d\mu(\zf) < \infty.
$$
\end{cor}

Finally, we anticipate a result from \cite{BSSW3} which appears very natural, but 
whose proof is quite subtle.

\begin{pro}\label{pro:dens} Let $\Omega = \Omega_v$ or $\Omega = \Sf_v^o \subset \HT$  
($v \in V(\T)$). 
Then for any starting point $\zf \in \Omega\,$, the exit 
measure $\mu_{\zf}^{\Omega}$ has a continuous, strictly positive
density with respect to Lebesgue measure on the finitely many bifurcation lines
that make up $\partial \Omega$.

The analogous statement holds on ``sliced'' hyperbolic plane.
\end{pro}

\section{Rate of escape and convergence to the boundary at infinity}\label{sec:boundary}

\begin{thm}\label{thm:escape}
In the natural metric of $\HT$, the Brownian motion $(X_t)$ on $\HT$
generated by $\Lap_{\al,\beta}$ has the following rate of escape.
$$
\lim_{t \to \infty} \frac{1}{t} \,\dist_{\HT}(X_t,X_0) = 
|\ell(\al,\beta)| \quad \text{almost surely, where}\quad 
\ell(\alpha,\beta) = \frac{\log \qq}{\Ex(\tau)} \,\frac{\rha-1}{\rha+1}\,,
$$
with $\rha$ and $\Ex(\tau)$ given by Proposition \ref{pro:tau}(d) 
and \eqref{eq:Etau}, respectively. 
\end{thm}

The proof of this theorem will go hand in hand with the one of Theorem 
\ref{thm:Brownianlimit} below, concering convergence of $(X_t)$ to the boundary.

The tree $\T$ has its natural geometric compactification $\wh\T$ with boundary
at infinity $\bd \T = \bd^*\T \cup \{ \om \}$, see Figure~2. Analogously,
the hyperbolic plane $\Hb$ has its standard hyperbolic compactification $\wh\Hb$ 
with boundary $\bd \Hb = \bd^*\Hb \cup \{ \binfty \}$, where $\bd^*\Hb = \R$, see
Figure~3. Since $\HT$ is a topological subspace of $\Hb \times \T$, we can compactify
it as follows.

\begin{dfn}\label{def:compact}
The \emph{geometric compactification} $\wh\HT$ of $\HT$ is the closure of $\HT$ in
the compact space $\wh \Hb \times \wh \T$. The \emph{geometric boundary} at infinity
of $\HT$ is
$$
\bd\HT = \wh\HT \setminus \HT\,.
$$
\end{dfn}

The boundary consists of the following five pieces.
\begin{equation}\label{eq:pieces}
\bd\HT= \bigl( \{\binfty\}\times \bd^*\T \bigr) \cup \bigl( \bd^*\Hb \times \{\om\} \bigr)
\cup \bigl( \{\binfty\} \times \T \bigr) \cup \bigl( \Hb \times \{\om\} \bigr)
\cup \{(\binfty,\om) \}  \,.
\end{equation}
For a better understanding (and future use), we describe convergence to the boundary.

\begin{imp}\label{imp:convergence}
Consider a sequence $\zf_n=(z_n,w_n)$ in $\HT$, with $z_n=x_n+\im y_n$.
\begin{itemize}
\item[(a)\quad] $\zf_n \to (\binfty,\xi) \in \{\binfty\}\times \bd^*\T$ if 
$w_n \to \xi$ in $\wh\T$, in which case necessarily $z_n \to \binfty$.
\item[(b)\quad] $\zf_n \to (\zeta,\om) \in \bd^*\Hb \times \{\om\} $ if 
$z_n \to \zeta$ in $\wh\Hb$, that is, $x_n \to \zeta$ and $y_n \to 0$ as
sequences in $\R$. In this case necessarily $w_n \to \om$.
\item[(c)\quad] $\zf_n \to (\binfty,w) \in \{\binfty\} \times \T$ if 
$w_n \to w$ in $\T$ and $z_n \to \binfty$ in $\wh\Hb$, that is, $|x_n| \to +\infty$ and 
$y_n \to \qq^{\hor(w)}$ as sequences in $\R$. 
\item[(d)\quad] $\zf_n \to (z,\om) \in \Hb \times \{\om\}\;$ if 
$\;z_n \to z\;$ in $\;\Hb\;$ and $\;w_n \to \om\;$ in $\;\wh \T\,$, that is, 
$d(o,w_n \cf o) \to +\infty$ and $\hor(w_n) \to \log_{\qq}(\IM z)$.
\item[(e)\quad] $\zf_n \to (\binfty,\om)$ if $z_n \to \binfty$ and $w_n \to \om$.
In this case, up to passing to a sub-sequence, we may assume in addition that
there is $\tau \in [-\infty\,,\,+\infty]$ such that $\hor(w_n) \to \tau$ and
$y_n \to \qq^\tau \in [0\,,\,+\infty]$. (Each value $\tau$ can be attained in
the limit by some sequence $\zf_n$.) 
\end{itemize}
\end{imp}

\begin{thm}\label{thm:Brownianlimit}
In the topology of $\wh\HT$, the Brownian motion $X_t=(Z_t,W_t)$ on $\HT$
generated by $\Lap_{\al,\beta}$ converges almost 
surely to a boundary-valued limit random variable 
$X_{\infty} = (Z_{\infty},W_{\infty})$.
Writing $\nu_{\zf}$ for its distribution when $X_0=\zf$, we have the following. 
\begin{itemize} 
\item[(i)] If $\ell(\al,\beta) > 0$ then 
$X_{\infty} \in \{\binfty\} \times \bd^*\T$, and all of the latter set is charged
by $\nu_{\zf}$.
\item[(ii)] If $\ell(\al,\beta) < 0$ then 
$X_{\infty} \in \bd^*\Hb \times \{\om\}$, and all of the latter set is charged
by $\nu_{\zf}$.
\item[(iii)] If $\ell(\al,\beta) = 0$ then
$X_{\infty} = (\binfty,\om)$, a deterministic limit. 
\end{itemize}
\end{thm}

The most useful tool for proving the last two theorems is the notion of
regular sequences of {\sc Kaimanovich}~\cite{Kai}, which we formulate here just for 
hyperbolic plane and tree.

\begin{dfn}\label{dfn:regular} Let $\mathbb X = \Hb$ or $\mathbb X= \T$. 
A sequence $(z_n)$ in $\mathbb X$ is called regular with \emph{rate} 
$r \ge 0$ if there is a geodesic ray
$(\pi_t)_{t \ge 0}$ in $\mathbb X$ [that is, $\dist_{\mathbb X}(\pi_t,\pi_s) 
= |t-s|$ for all $s,t \ge 0$] such that
$$
\dist_{\mathbb X}(x_n\,, \pi_{rn})/n \to 0 \quad \text{as}\; n \to \infty\,.
$$ 
\end{dfn}

The following was shown in \cite{Kai}.

\begin{lem}\label{lem:regH}
A sequence $(z_n)$ in $\Hb$ is regular if and only if there is
$b \in \R$ such that 
$$ 
\log \IM(z_n)/ n \to b \AND \dist_{\Hb}(z_{n+1}\,,z_n)/n \to 0.
$$
In this case, $r = |b|$ and $\dist_{\Hb}(z_n,z_0)/n \to r$.

Furthermore, if $b > 0$ then $z_n \to \binfty$ in the topology of $\wh\Hb$, 
while if $b < 0$ then there is some $\zeta \in \partial^* \Hb$ such that
$z_n \to \zeta$ in the topology of $\wh\Hb$. 
(There is no general statement of this form when $b=0$.)
\end{lem}

The analogue for trees was proved in \cite{CKW}.

\begin{lem}\label{lem:regT}
A sequence $(w_n)$ in $\T$ is regular if and only if there is
$b \in \R$ such that 
$$ 
\hor(w_n)/ n \to b \AND \dist_{\T}(w_{n+1}\,,w_n)/n \to 0.
$$
In this case, $r = |b|$ and $\dist_{\T}(w_n,w_0)/n \to r$.

Furthermore, if $b > 0$ then $w_n \to \varpi$ in the topology of $\wh\T$, 
while if $b < 0$ then there is some $\xi \in \partial^* \T$ such that
$w_n \to \xi$ in the topology of $\wh\T$. 
(Again, there is no general statement of this form when $b=0$.)
\end{lem}

Before embarking on the proofs of the above two theorems, we also
need the following.
\begin{lem}\label{lem:LLN} \hspace{20pt}
$\displaystyle
\lim_{t \to \infty} Y_t/t = \ell(\al,\beta)/\log \qq \quad\text{almost surely,}
$
where (recall) $Y_t= \pi^{\R}(X_t)$. 
\end{lem}

\begin{proof} Corollary \ref{cor:RW} and the Law
of Large Numbers imply that $\frac{1}{n}Y_{\tau(n)}\to \frac{\rha-1}{\rha + 1}$
almost surely.  Again by the law of large numbers, Proposition \ref{pro:tau}
tells us that $\tau(n)/n \to \Ex(\tau)$ almost surely.
Combining these two facts, we get that 
$Y_{\tau(n)}/\tau(n) \to \frac{\rha-1}{\rha + 1}\big/\Ex(\tau)$
almost surely. 
Given $t > 0$, let the random $\nb_t \in \N$ be as in \eqref{eq:nt}. Then $\nb_t \to \infty$ 
and $\tau(\nb_t)/t \to 1$ almost surely, as $t \to \infty$. By construction, 
$Y_t$ lies between $Y_{\tau(\nb_t)}$ and $Y_{\tau(\nb_t+1)}\,$, which differ by $1$.
Therefore the almost sure limit
$$
\lim_{t \to \infty} \frac{Y_t}{t} 
= \lim_{t \to \infty} \frac{Y_{\tau(\nb_t)}}{t} 
= \lim_{t \to \infty} \frac{Y_{\tau(\nb_t)}}{\tau(\nb_t)}\,\frac{\tau(\nb_t)}{t}
$$
exists and has the proposed value.
\end{proof}

Let us now consider the process $(W_t)$ on $\T$.

\begin{pro}\label{pro:Wtaun}
Let $\rha$ be as in Proposition \ref{pro:tau}(d) and $\ell(\al,\beta)$ as in 
Theorem \ref{thm:escape}. Then
$$
\lim_{t \to \infty} \frac{1}{t}\,\dist_{\T}(W_t\,,W_0) = 
\frac{1}{\log\qq}\,|\ell(\al,\beta)|\quad\text{almost surely.}
$$
\noindent
If $\ell(\al,\beta) \le 0$ ($\!\!\iff \rha \le 1$) then for any starting
point $w \in \T$, 
$$
\lim_{t \to \infty} W_t = \varpi \quad \text{almost surely in the topology of}\;\; 
\wh \T\,.
$$
If $\ell(\al,\beta) > 0$ ($\!\!\iff \rha > 1$) then there is a $\partial^*\T$-valued
random variable $W_{\infty}$ such that for any starting
point $w \in \T$, we have almost surely that
$$
\lim_{t \to \infty} W_t = W_{\infty} \quad \text{in the topology of}\; \;
\wh \T\,.
$$
In this case, let $\nu_w^{\partial \T}$ be the distribution of $W_{\infty}\,$,
given that $W_0 = w$. This is a probability measure that has a strictly positive, continuous, 
bounded density with respect to the ``Lebesgue'' measure $\la^*$ on $\partial^*\T$
explained in Remark \ref{rmk:density-tree}.
\end{pro}

\begin{proof} Consider first the random walk $(W_{\tau(n)})$ on $V(\T)$. 
As $\dist_{\T}(W_{\tau(n+1)}\,,W_{\tau(n)}) = 1$, lemmas \ref{lem:LLN}
and \ref{lem:regT} yield that the sequence $\bigl(W_{\tau(n)}\bigr)$ is almost
surely regular. We obtain that first of all,
$$
\frac{1}{n}\,\dist_{\T}(W_{\tau(n)}\,,W_0) \to \left|\frac{\rha-1}{\rha+1}\right|
\quad\text{almost surely.}
$$ 
The proof now proceeds as the one of Lemma \ref{lem:LLN}: with $\nb_t$ 
as in \eqref{eq:nt}, we have that $W_t$ lies on the edge between
$W_{\tau(\nb_t)}$ and $W_{\tau(\nb_t+1)}$, whence 
$\dist_{\T}(W_t\,,W_{\tau(\nb_t)}) \le 1$. Therefore 
$$
\lim_{t\to \infty}\frac{\dist_{\T}(W_t\,,W_0)}{t} 
= \lim_{t\to \infty}\frac{\dist_{\T}(W_{\tau(\nb_t)}\,,W_0)}{\nb_t}
\,\frac{\nb_t}{\tau(\nb_t)}
\,\frac{\tau(\nb_t)}{t}
= \left|\frac{\rha-1}{\rha+1}\right| \frac{1}{\Ex(\tau)}\,,
$$
as proposed. 

Second, again by Lemma \ref{lem:regT}, $(W_{\tau(n)})$ converges 
a.s. to $\varpi$, when $\rha < 1$.

When $\rha>1$, it converges a.s. to a $\partial^*\T$-valued random variable 
$W_{\infty}$. Using the formulas 
that are displayed in \cite[Proposition 9.23]{Wo-markov}, one can compute  the limit 
distribution $\nu_v^{\partial \T}$ of that random walk, when $W_0 = v \in V(\T)$.
Explicit computations can be found in \cite{Wlamp}. One sees that 
$\nu_v^{\partial T}$ has the stated properties. In particular, it carries no 
point mass and is supported by the whole of $\partial^*\T$
(or equivalently, $\partial \T$).

If we replace the starting point $v \in V(\T)$ by a starting point $w$ that lies
in the interior of some edge $[v^-,v]$ then the process starting at $w$ also
must converge to $\partial^*\T$, and we have
$$
\nu_w^{\partial \T} = \Prob[W_{\tau(1)} = v \mid W_0=w]\, \nu_v^{\partial \T} 
+ \Prob[W_{\tau(1)} = v^- \mid W_0=w]\, \nu_{v^-}^{\partial \T}\,.
$$ 
We still have to show that $W_t \to \varpi$ almost surely, when 
$\rha=1$. This is obtained by the following simple argument. Being a transient 
nearest neigbour random walk, $(W_{\tau(n)})$ must converge almost surely
to some random end of $\T$, see \cite[Theorem 9.18]{Wo-markov}. 
But the projection $\hor(W_{\tau(n)}) = Y_{\tau(n)}$ is a recurrent random walk
on $\Z$, when $\rha =1$. Thus, there is a random subsequence $(n')$ along which
$\hor(W_{\tau(n')}) = 0$. This subsequence must converge to $\varpi$,
whence $\varpi$ is the limit of the entire sequence.
\end{proof}

In fact, the last proposition provides the simplest class of cases to
which the results of \cite{CKW} apply (but explaining how to apply those
general results would consume more space and energy than the above direct
arguments.) 
We next want to present the analogous proposition concerning the process $(Z_t)$
on $\Hb$. Recall that we can interpret the random walk $(Z_{\tau(n)})$ on $\LH$ 
as a right random walk on the group $\Aff(\Hb,\qq)$ which is identified with $\LH$
via \eqref{eq:identif}. With this identification, the law of that random walk 
is the probability $\wt\mu$ of \eqref{eq:wtmu}. We know that in the notation of
the group operation, the increments 
$Z_{\tau(n-1)}^{-1}Z_{\tau(n)}\,$, $n \ge 2$, are i.i.d.  
$\Aff(\Hb,\qq)$-valued random variables with common distribution $\wt\mu$, 
so that they can be written
as random affine transformations $\bigl(\begin{smallmatrix}A_n & B_n \\ 0 & 1
                      \end{smallmatrix}\bigr)$, where 
$A_n = \qq^{Y_{\tau(n)}-Y_{\tau(n-1)}}$; the associated transformation
of $\Hb$ is $z \mapsto A_nz + B_n\,$.
While $A_n$ only takes the two values $\qq$ and 
$1/\qq$, the common distribution of the real random variables $B_n$ has
a continuous density with respect to Lebesgue measure by Proposition \ref{pro:dens}. 
By Corollary \ref{cor:mumom}, $B_n$ satisfies an 
exponential moment condition.

\begin{pro}\label{pro:Ztaun}\hspace{15pt}
$\displaystyle
\lim_{t \to \infty} \frac{1}{t}\,\dist_{\Hb}(Z_t\,,Z_0) = 
|\ell(\al,\beta)|\quad\text{almost surely.}
$\\[5pt]
\noindent
If $\ell(\al,\beta) \ge 0$ ($\!\!\iff \rha \ge 1$) then for any starting
point $z \in \Hb$, we have almost surely that
$$
\lim_{t \to \infty} Z_t = \binfty \quad \text{almost surely in the topology of}\; \;
\wh \Hb\,.
$$
If $\ell(\al,\beta) < 0$ ($\!\!\iff \rha < 1$) then there is a 
random variable $Z_{\infty}$ taking values in $\partial^*\Hb = \R$
such that for any starting point $z \in \Hb$, we have almost surely that
$$
\lim_{t \to \infty} Z_t = Z_{\infty} \quad \text{in the topology of}\; \;
\wh \Hb\,.
$$
In this case, let $\nu_z^{\partial \Hb}$ be the distribution of $Z_{\infty}\,$,
given that $Z_0 = z$. This is a probability measure on $\partial^*\Hb \equiv \R$
that has a continuous, strictly positive density with respect to Lebesgue 
measure.
\end{pro}

\begin{proof} 
By Corollary \ref{cor:inter},
$$
\frac{1}{n}\dist_{\Hb}(Z_{\tau(n+1)}\,,Z_{\tau(n)})/n \to 0 
\quad\text{almost surely.} 
$$
By \ref{lem:LLN}, 
\begin{equation}\label{eq:rate}
\frac{1}{n}\log \IM(Z_{\tau(n)})= \frac{\log \qq}{n}\,Y_{\tau(n)} 
\to \log \qq \,\frac{\rha-1}{\rha + 1}\quad\text{almost surely.} 
\end{equation}
Thus, by Lemma \ref{lem:regH}, the sequence $(Z_{\tau(n)})$ is 
almost surely regular in $\Hb$, with
rate $\log \qq \big|\frac{\rha-1}{\rha + 1}\big|$. When $\rha > 1$ it converges
to $\binfty$ in the topology of $\wh\Hb$, while when $\rha < 1$, it converges
in that topology to a random element of $\partial^*\Hb$. 

\smallskip

The more difficult situation is the one where the rate of the sequence is $0$.
In that case, it was proved by 
{\sc Brofferio}~\cite{Br} that $Z_{\tau(n)} \to \binfty$ almost surely in the
topology of $\wh \Hb$. This is not yet enough to guarantee that also
$Z_t \to \binfty$ almost surely. We take inspiration from \cite{Br}. 
For any $g \in \Aff(\Hb,\qq)$, a neighbourhood base of $\binfty$ in $\wh \Hb$
is given by the collection of all sets $\wh \Hb \setminus g^{-1}V_r\,$,
where 
$$
V_r = \{ z = x+\im y : |x| \le r \;\text{and}\; 0 \le y \le \qq^r\}\,,
\quad r \in \N\,.
$$
Our argument will not depend on the starting point, but only on what
happens from time $\tau(1)$ onwards. Thus, we may assume that $Z_0 \in \LH$, 
which can be identified with $\Aff(\Hb,\qq)$.  
We know from \cite{Br} that for any $r$ and for any starting point in $\LH$, we have almost surely that 
$Z_{\tau(n)} \in \Hb \setminus V_r$ for all but finitely many $n$.
Equivalently, for starting point $\im$ and for some $g \in \Aff(\Hb,\qq)$, for any $r$ 
we have $Z_{\tau(n)} \in \Hb \setminus g^{-1}V_r$ for all but finitely many $n$ .

Thus, we need an element $g \in \Aff(\Hb,\qq)$ such that with probability $1$, in between the times 
$\tau(n)$ and $\tau(n+1)$, the process $(X_t)$ 
does not enter into $g^{-1}V_r\,$, if $n$ is sufficiently
large. This will follow from the Borel-Cantelli Lemma after showing that
\begin{equation}\label{eq:BorCant}
\sum_{n=1}^{\infty}\, \Prob_{\im}\!\left[{ Z_{\tau(n)}\in\Hb \setminus g^{-1}V_r\;,\;\,
Z_t \in g^{-1}V_r \; \atop \text{for some $t$ with}\;
\tau(n) < t < \tau(n+1)}\right] \,< \infty\,.
\end{equation}
Again, we use the identification \eqref{eq:identif} of $\LH$ with 
$\Aff(\Hb,\qq)$ and consider the potential measure 
$\mathcal U = \sum_{n=0}^{\infty} \wt \mu^{(n)}$, 
where $\wt \mu^{(n)}$ is the $n$th convolution power of the measure 
$\wt\mu$ on $\Aff(\Hb,\qq)$. By transience of $(Z_{\tau(n)})\,$, this
$\mathcal U$ is a Radon measure on $\Aff(\Hb,\qq) \equiv \LH$. For $z \in \LH$, let 
$$
f_r(z) = \uno_{\LH \setminus V_r}(z)\,
\Prob_z[  Z_t \in V_r \; \text{for some $t$ with}\; 0 < t < \tau(1) ].
$$
Then for any $g \in \Aff(\Hb,\qq)$,
$$
\sum_{n=1}^{\infty}\, \Prob_{\im}\!\left[{ Z_{\tau(n)}\in\Hb \setminus g^{-1}V_r\;,\;\,
Z_t \in g^{-1}V_r \; \atop \text{for some $t$ with}\;
\tau(n) < t < \tau(n+1)}\right] \,
= \int_{\LH} f_r(gz) \, d\,\mathcal U(z)\,.
$$ 
Let $z = b + \im \qq^m \in \LH \setminus V_r\,$, with $m \in \Z$ and $b\in \R$. 
Write $z=g_z\im$, where 
$g_z=\bigl(\begin{smallmatrix} \pp^m & b \\ 0 & 1 \end{smallmatrix}
\bigr)\in \Aff(\Hb,\qq)$. Then
$$
f_r(z) = \uno_{\LH \setminus g_z^{-1}V_r}(\im)\,
\Prob_{\im}[  Z_t \in g_z^{-1}V_r \; \text{for some $t$ with}\; 0 < t < \tau(1) ].
$$
We have 
$$
g_z^{-1}V_r = \{ x + \im y : |x+\qq^{-m}b| \le \qq^{-m} r \;
\text{and}\; 0 \le y \le \qq^{r-m}\}.
$$
We must have $\im \in \LH \setminus g_z^{-1}V_r\,$. Starting at $\im$, 
the process $(Z_t)$ does not leave $\Sf_0 \cup \Sf_1$
before time $\tau(1)$. Compare with Figure~6. Thus, in order to be able to 
enter into $g_z^{-1}V_r$ before
that time, we must have $r-m \ge 0$; otherwise $f_r(z)=0$. 

Suppose that we do 
have $r-m \ge 0$, and that $\im$ stays to the left of $g_z^{-1}V_r$, so that 
$-r-b > 0$. 
Then in order to enter into $g_z^{-1}V_r$ before $\tau(1)$, the process must cross 
the vertical line where $x = -\qq^{-m}(r+b)$. Setting 
$k= \lfloor -\qq^{-m}(r+b)\rfloor$ (next lower integer), this means that 
$Z_t$ must pass through the segment $J_k$ of Figure~6. By Proposition 
\ref{pro:sideways}, resp. \eqref{eq:occur} in its proof, $f_r(z) \le \rho^k$. 
Analogously, if  $\im$ stays to the right of $g_z^{-1}V_r\,$, which means that 
$r-b < 0$, then $f_r(z) \le \rho^k$, where $k = \lfloor \qq^{-m}(b-r)\rfloor$.
Setting $\la =-\log \rho$, we find that
$$
f_r(b + \im \qq^m) \begin{cases}= 0\,,&\quad\text{if}\; m > r \;\text{or}\; |b| < r\\
\le \exp\Bigl(-\la\bigl(\qq^{-m}(|b|-r)+1\bigr)\Bigr) 
&\quad\text{if}\; m \le r \;\text{and}\; |b| \ge r.
\end{cases}
$$
The \emph{right} Haar measure on $\Aff(\Hb,\qq)  \equiv \LH$ is  one-dimensional
Lebesgue measure on each of the lines $\Lf_k\,$, compare with
\eqref{eq:modularAffRq}. Thus, the integral
of $f_r$ with respect to right Haar measure is 
$$
\sum_{m \le r} \int_{|b| \ge r} f_r(b + \im \qq^m)\,db < \infty\,.
$$
\cite[Lemma 1]{Br} yields that in this case, 
$\int_{\LH} f_r(gz) \, d\,\mathcal U(z) < \infty$
for $dg$-almost all $g \in \Aff(\Hb,\qq)$. This is true for all $r \in \N$.
Thus, there is some fixed $g \in \Aff(\Hb,\qq)$ such that the last integral
is finite for every $r\in \N$. For this $g$, 
\eqref{eq:BorCant} holds for every $r\in \N\,$, so that $Z_t \to \binfty$
almost surely.
 
\smallskip

We finally have to explain that in the case $\rha < 1$, the limit random variable on
$\partial^*\Hb$ has a distribution with continuous, positive density
with respect to Lebesgue measure.

Let us write $Z_{\tau(1)} 
= \bigl(\begin{smallmatrix} A_1 & B_1 \\ 0 & 1 \end{smallmatrix}\bigr)$, 
which is independent of the other
$\bigl(\begin{smallmatrix} A_n & B_n \\ 0 & 1 \end{smallmatrix}\bigr)$ but 
does in general not have the same distribution.
We know from Proposition \ref{pro:dens} that for arbitrary starting point $z \in \Hb$, 
the distribution of $B_1$ has a continuous density with respect to Lebesgue measure.
We then have
$$
Z_{\tau(n)} = 
\begin{pmatrix} A_1 \cdots A_n &\; \sum_{k=1}^n A_1 \cdots A_{k-1}B_k \\ 0 & 1 
\end{pmatrix}.$$
When $\rha < 1$, it is very well known and quite easy
to verify that in $\R$, the upper right matrix element 
of $Z_{\tau(n)}$ converges almost surely to 
$$
Z_{\infty} = \sum_{k=1}^{\infty} A_1 \cdots A_{k-1}B_k\,,
$$
as $n \to \infty\,$. Recalling the identification \eqref{eq:identif}, we see that
this is the limit of $Z_{\tau(n)}$ in $\wh \Hb$, since 
$A_1 \cdots A_n \to 0$ almost surely. 
It is now easy to verify  that along with all the $B_n$ 
(including $B_1$), for arbitrary starting point also the distribution of
$Z_{\infty}$ has a continuous density with respect to Lebesgue measure on $\R$.
\end{proof}

\begin{rmk}\label{rmk:gap} Our result on almost sure convergence of $(Z_t)$
to $\binfty$ in the critical case $\rha=1$ also applies to Brownian motion
with vertical drift on $\Hb$ without any bifurcation lines. Indeed, this corresponds
just to the case when $\beta\pp =1$.  This closes a small gap left open in the 
proof of \cite[Proposition 4.2.(iii)]{BSW}, concerning the passage
from discrete to continuous time.
\end{rmk}

\begin{proof}[\bf Proof of theorems \ref{thm:escape} and \ref{thm:Brownianlimit}]
Theorem \ref{thm:escape} regarding the rate of escape of $(X_t)$ now follows by
combining the inequalities of Proposition \ref{pro:metric} with the 
rates of escape of $(Y_t)$, $(W_t)$ and $(Z_t)$, as provided by Lemma 
\ref{lem:LLN} and propositions \ref{pro:Wtaun} and \ref{pro:Ztaun}, respectively.

\smallskip

Theorem \ref{thm:Brownianlimit} follows by combining those two propositions 
with the description \eqref{imp:convergence}
of convergence to the boundary in treebolic space.
\end{proof}
Theorem \ref{thm:Brownianlimit} provides the following, which (as mentioned)
was only indicated in \cite{BSSW}. 

\begin{cor}\label{cor:transience}
The processes $(X_t)$ on $\HT$, $(Z_t)$ on $\Hb$ and $(W_t)$ on $\T_{\pp}$ 
($\pp \ge 2$), as defined in Proposition \ref{pro:projections}, are transient.
\end{cor}

\section{Central limit theorem}\label{sec:CLT}

The proof of a CLT for $\dist(X_t,X_0)$ ($t \to \infty$) depends
significantly on the sign of the drift $\ell(\al,\beta)$.
It will follow from the CLT for the random walk $(X_{\tau(n)})$.
Here we shall work with $\dist(X_t,\of)$ instead of $\dist(X_t,X_0)$,
which makes no difference, as we divide by $\sqrt t$. In any case, before
that we need the CLT for the vertical component $Y_t$ of $X_t\,$.

\begin{lem}\label{lem:CLTY} With $\Var(Y)$ and $\Var(\tau)$
as in Proposition \ref{pro:tau}, set
$$
\sigma^2 = \sigma^2(\al,\beta) =
\frac{1}{\Ex(\tau)} \Var(Y)
+ \frac{\ell(\al,\beta)^2}{\Ex(\tau)\log^2 \qq}\Var(\tau).
$$
Then
$$
\frac{1}{\sqrt t}\Bigl(Y_t - t\, \frac{\ell(\al,\beta)}{\log\qq}\Bigr) \to N(0,\sigma^2)
\quad \text{in law, as}\; t \to \infty\,.
$$
\end{lem}

\begin{proof}
The $\R^2$-valued random variables 
$\bigl(Y_{\tau(n)} - Y_{\tau(n-1)}\,,\tau(n) - \tau(n-1)\bigr)_{n\ge 2}$ 
are i.i.d., see Proposition \ref{pro:tau}. By the two-dimensional CLT,
\begin{equation}\label{eq:CLT2}
\frac{1}{\sqrt{n}}\Bigl(Y_{\tau(n)} - n\, \frac{\rha-1}{\rha+1}\,,\,
\tau(n) - n \Ex(\tau)\Bigr) \to N(0,\Sigma^2) \quad\text{in law,}
\end{equation}
where $\Ex(\tau)$ is as in Proposition \ref{pro:tau}(e) and $N(0,\Sigma^2)$
is the two-dimensional normal distribution with mean vector $0$ and 
$\Sigma^2$ is the covariance matrix of 
$\bigl(Y_{\tau(2)} - Y_{\tau(1)}\,,\tau(2) - \tau(1)\bigr)$, which is just the
diagonal matrix with diagonal entries $\Var(Y)$ and  $\Var(\tau)$.

As in the proof of Lemma \ref{lem:LLN}, with the $\nb_t$ of \eqref{eq:nt},
we know that 
\begin{equation}\label{eq:nnt:t}
\frac{\nb_t}{t} = \frac{\nb_t}{\tau(\nb_t)}\frac{\tau(\nb_t)}{t}  \to
\frac{1}{\Ex(\tau)} \quad \text{almost surely, as}\; t \to \infty\,,
\end{equation}
and that $\bigl|Y_t - Y_{\tau(\nb_t)}\bigr| < 1$.
Now we decompose
$$
\begin{aligned}
\frac{Y_t - t\, \frac{\ell(\al,\beta)}{\log\qq}}{\sqrt t}
& = \frac{Y_t - Y_{\tau(\nb_t)}}{\sqrt t}
+ \sqrt{\frac{\nb_t}{t}}\cdot
\frac{Y_{\tau(\nb_t)} - \nb_t\, \frac{\rha-1}{\rha+1}}{\sqrt \nb_t}\\
&\quad - \frac{\ell(\al,\beta)}{\log \qq}\cdot\frac{t-\tau(\nb_t)}{\sqrt t}
- \sqrt{\frac{\nb_t}{t}}\frac{\ell(\al,\beta)}{\log \qq}\cdot
\frac{\tau(\nb_t) - \nb_t \Ex(\tau)}{\sqrt{\nb_t}}\,.
\end{aligned}
$$
The first term of the sum on the right hand side
tends to $0$ because $0 \le Y_t - Y_{\tau(\nb_t)} < 1$ almost surely.
The third term  tends to $0$ almost surely, because 
$$
\frac{t-\tau(\nb_t)}{\sqrt t} \le \frac{\tau(\nb_t+1)-\tau(\nb_t)}{\sqrt{\nb_t}}
\sqrt{\frac{\nb_t}{t}},
$$
and $\bigl(\tau(n+1)-\tau(n)\bigr)/\sqrt n \to 0$  
by Proposition \ref{pro:tau}(d).
Also, we know that $\nb_t/t \to 1/\Ex(\tau)$ almost
surely. Hence, 
$$
\frac{Y_t - t\, \frac{\ell(\al,\beta)}{\log\qq}}{\sqrt t}
\simlaw  \frac{1}{\sqrt{\Ex(\tau)}}\cdot
\frac{Y_{\tau(\nb_t)} - \nb_t\, \frac{\rha-1}{\rha+1}}{\sqrt \nb_t} 
- \frac{\ell(\al,\beta)}{\sqrt{\Ex(\tau)}\log \qq}\cdot
\frac{\tau(\nb_t) - \nb_t \Ex(\tau)}{\sqrt{\nb_t}}
$$
as $t \to \infty$. It follows from  \eqref{eq:CLT2} that this converges in
law to the centred normal distribution with variance
$\sigma^2(\al,\beta)$,
as proposed.
\end{proof}

\begin{lem}\label{lem:reduce}
\emph{(a)} If $\ell(\al,\beta) > 0$ then
$$
\limsup_{t \to \infty} 
\Bigl(\dist_{\HT}(X_t\,,o) - \dist_{\Hb}(Z_t\,,\im)\Bigr) 
< \infty\quad\text{almost surely.}
$$
\emph{(b)} If $\ell(\al,\beta) < 0$ then
$$
\limsup_{t \to \infty} \Bigl(\dist_{\HT}(X_t\,,\of) 
- (\log \qq) \dist_{\T}(W_t,o)\Bigr) < \infty\quad\text{almost surely.}
$$
(The two appearing differences are always non-negative.)
\end{lem}

\begin{proof}
(a) By Proposition \ref{pro:Wtaun}, $W_t \to W_{\infty} \in \partial^*\T$
almost surely. Therefore $\hor(o \cf W_t) \to \hor(o\cf W_{\infty})$
a.s., that is, the two (finite !) random variables coincide from some random 
$t_0$ onwards, 
and in particular $\hor(W_t) = Y_t\ge 0$ for $t \ge t_0\,$.
By \eqref{eq:tree-conf}, 
$$
\dist_{\T}(W_t,o) = \hor(W_t) - 2\hor(o\cf W_t) = Y_t - 2\hor(o\cf W_{\infty})
\quad  \text{for all}\; t \ge t_0\,,
$$
and for those $t$, the first inequality of Proposition \ref{pro:metric} yields
$$
\dist_{\HT}(X_t\,,\of) \le \dist_{\Hb}(Z_t\,,\im) - 2(\log\qq)\hor(o\cf W_{\infty}).    
$$
(We note here that $\hor(o\cf W_{\infty}) \le 0$.)  
This yields (a). 

\smallskip

(b) This time, we use Proposition \ref{pro:Ztaun} and get that 
$\im \wedge Z_t \to \im \wedge Z_{\infty} \in \Hb$ (a.s. convergence in $\wh\Hb$).  
Therefore, by \eqref{eq:hyp-conf}, 
$$
\limsup_{t \to \infty}
\Bigl| \dist_{\Hb}(Z_t\,,\im) - \Bigl( 2\log\bigl(\IM (\im \wedge Z_{\infty})\bigr)
 - \log \IM Z_t \Bigr) \Bigr|  < \infty \quad\text{almost surely.}
$$
Note that $\log \IM Z_t < 0$ if $t$ is sufficiently
large. Thus, in the same way as in (a), Proposition \ref{pro:metric} yields
statement (b).
\end{proof}

We now consider $(X_{\tau(n)})$.
The group $\Af$ acts transitively on the set $\LT$ defined in \eqref{eq:LT}
of all bifurcation lines in $\HT$. In part (2) of the proof of Theorem
\ref{thm:isogroup}, we have introduced the coordinates $[b,\ga]$ for the
elements of $\Af$. In the same way, it will be useful to use coordinates
$[x,v]$ for the elements of $\LT$, such that $[x,v]$ is the point on $\Lf_v$ 
with horizontal coordinate $x$, that is, $[x,v] = (x + \im \qq^{\hor(v)},v)$
in the notation of \eqref{eq:treebolicdef}. In these coordinates, the natural
$\Af$-invariant measure on $\LT$ is given by $d[x,v] =
q^{-\hor(v)}dx\,d_{\sharp}v\,$,
where $dx$ is standard Lebesgue measure and $d_{\sharp} v$ is the counting
measure on $V(\T)$. 

By (natural) abuse of notation, we also write $\pi^{\Hb}$ and $\pi^{\T}$ for
the projections $(g,\ga) \mapsto g$ and $(g,\ga) \mapsto \ga$ 
from $\Af$ onto $\Aff(\Hb,\qq)$ and $\Aff(\T)$, respectively.
This refers to the notation used 
in the statement of Theorem \ref{thm:isogroup},
while in the $[b,\ga]$-coordinates, $\pi^{\T}[b,\ga] = \ga$ and
$\pi^{\Hb}[b,\ga] 
=\bigl(\begin{smallmatrix}\qq^{\Phi(\ga)} & b\\ 0 & 1\end{smallmatrix}\bigr)$ 
as an affine transformation.

Now recall from \S \ref{sec:rw} that $(X_{\tau(n)})$ is a Markov chain on $\LT$
whose transition probabilities are given by the probability measure $\mu$, see
\eqref{eq:degmu}. By Proposition \ref{pro:dens}, $\mu$ has
a continuous density, which we denote by $f_{\mu}\,$, with respect to 
$d[x,v]$. The projection $\wt \mu$ also has a continuous density
on $\LH$ with respect to the $\Aff(\Hb,\qq)$-invariant measure which is analogous
to $d[x,v]$. Furthermore, we note that for $v \in V(\T)$,  
$$
\int_{\R} f_{\mu}[x,v]\,dx = p(o,v)\,,
$$
the transition probabilities of $(W_{\tau(n)})$ that appeared in Corollary 
\ref{cor:RW}(b). We now lift $\mu$ to a probability measure $\mmu$ on the 
group $\Af\,$: it has density $\fb$ with respect to the Haar measure  
\eqref{eq:HaarAf} on $\Af\,$, where
$
\fb[b,\ga] = f_{\mu}[b,\ga o]\,.
$

We then can construct (on a suitable probability space) 
a sequence $(\XX_n)_{n \ge 1}$ of i.i.d. $\Af$-valued random variables
with common distribution $\mmu$, and the associated right random walk on
$\Af$,
$$
\RR_n = \XX_1 \, \XX_2 \cdots \XX_n\,, \; n \ge 0\,.
$$
The product is of course taken in the group $\Af$, and $\RR_0$ is the identity
of that group. The (simple) proof of part (i) of the following lemma is 
omitted; it follows  \cite[Lemma 3.1]{W-israel}, 
see also \cite[p. 5, Remarque 6]{GKR}. Statements (ii)--(iv) are immediate
consequences.

\begin{lem}\label{lem:gRno} \emph{(i)}
For any $\gf \in \Af$, the sequence
$(\gf\RR_n\of)$ is a realisation of the induced random walk 
$(X_{\tau(n)})_{n\ge 0}$ on
$\LT$ starting at $\gf\of$. That is, it is an $\LT$-valued Markov chain 
with transition probabilities \eqref{eq:transprob}.
\\[5pt]
\emph{(ii)} Via the identification \eqref{eq:identif} of $\LH$ with
$\Aff(\Hb,\qq)$, the random walk $\pi^{\Hb}(\RR_n)$ is a realisation
of the process $(Z_{\tau(n)})$ on $\LH$ starting at $\im$. 
\\[5pt]
\emph{(iii)} $R_n = \pi^{\T}(\RR_n)$ is a right random walk on the group 
$\Aff(\T)$, and the process $(R_no)_{n \ge 0}$ is a realisation of the 
random walk $(W_{\tau(n)})_{n\ge 0}$ on (the vertex set of) $\T$ as 
described in Corollary \ref{cor:RW}(b), with $R_0o = o$.
\\[5pt]
\emph{(iv)} With $\Phi$ as in \eqref{eq:modularAffTp}, 
the sequence $\bigl(\Phi(R_n)\bigr)_{n\ge 0}$ is a realisation of
the random walk $(Y_{\tau(n)})_{n \ge 0}$ on $\Z$ as described in Corollary
\ref{cor:RW}(a), with starting point $0$.
\end{lem}

\begin{thm}\label{thm:CLT-drift} If $\ell(\al,\beta) \ne 0$ and $\sigma^2$ is
as in Lemma \ref{lem:CLTY}
$$
\frac{1}{\sqrt t}\Bigl(\dist_{\HT}(X_t\,,\of) - t\, |\ell(\al,\beta)|\Bigr) 
\to N(0,\sigma^2) \quad \text{in law, as}\; t \to \infty\,.
$$
\end{thm}

\begin{proof} 
\emph{Case 1.}
$\;\ell(\al,\beta) > 0$. 
\\
Lemma \ref{lem:reduce}(a) tells us
 that we just have to consider  $\dist_{\Hb}(Z_t\,,\im)$. 
Using the same notation as before Proposition
\ref{pro:Ztaun}, we write $Z_{\tau(n)}^{-1}Z_{\tau(n)} 
=\bigl(\begin{smallmatrix}A_n & B_n \\ 0 & 1 \end{smallmatrix}\bigr)$ as
independent $\wt\mu$-distributed group elements of $\Aff(\Hb,\qq)$ for $n \ge 2$, 
as well as  $Z_{\tau(1)} =\bigl(\begin{smallmatrix}A_1 & B_1 \\ 0 & 1
                      \end{smallmatrix}\bigr)$, which is independent of the other ones (but
may have a different distribution, according to the starting point). 
The group inverses are 
$\bigl(\begin{smallmatrix}A_n & B_n \\ 0 & 1\end{smallmatrix}\bigr)^{-1}= 
\bigl(\begin{smallmatrix}1/A_n & \,-B_n/A_n \\ 0 & 1
                      \end{smallmatrix}\bigr)$. Since $A_n$ only takes values 
$\qq$ and $1/\qq$,  also $-B_n/A_n$ has exponential moments. 
Taking products in that group, 
$Z_{\tau(n)}= \bigl(\begin{smallmatrix}A_1 & B_1 \\ 0 & 1\end{smallmatrix}\bigr)\cdots 
 \bigl(\begin{smallmatrix}A_n & B_n \\ 0 & 1\end{smallmatrix}\bigr)$, so that  
$$
Z_{\tau(n)}^{-1} = \Bigl(\!\begin{smallmatrix}A_n & B_n \\[2pt] 0 & 1\end{smallmatrix}\!\Bigr)^{\!-1} 
\!\!\cdots 
\Bigl(\!\begin{smallmatrix}A_2 & B_2 \\[2pt] 0 & 1\end{smallmatrix}\!\Bigr)^{\!-1}
\Bigl(\!\begin{smallmatrix}A_1 & B_1 \\[2pt] 0 & 1\end{smallmatrix}\!\Bigr)^{\!-1}
\eqlaw 
\underbrace{\Bigl(\!\begin{smallmatrix}A_2 & B_2 \\[2pt] 0 & 1\end{smallmatrix}\!\Bigr)^{\!-1} \!\!\cdots 
\Bigl(\!\begin{smallmatrix}A_n & B_n \\[2pt] 0 & 1\end{smallmatrix}\!\Bigr)^{\!-1}
}_{\displaystyle =: Z^*_{\tau(n)}} 
\Bigl(\!\begin{smallmatrix}A_1 & B_1 \\[2pt] 0 & 1\end{smallmatrix}\!\Bigr)^{\!-1} \!\!. 
$$ 
Note that in $\R$, we have $A_1 \cdots A_n = \qq^{Y_{\tau(n)}}$.
Now
$(Z^*_{\tau_n})$ is again a right random walk on $\Aff(\Hb,\qq)$, and 
returning to the
identification with $\LH$, we have that
$$
\begin{gathered}
\dist_{\Hb}(Z^*_{\tau(n)}, Z^*_{\tau(n-1)})/n \to 0 \AND\\[3pt]
\frac{1}{n}\log \IM Z^*_{\tau(n)} = - \frac{(\log\qq)}{n} 
\bigl(Y_{\tau(n)} - Y_{\tau(1)}\bigr) \to -\log\qq\frac{\rha-1}{\rha+1}
\quad\text{almost surely}.
\end{gathered}
$$
Since the last limit is $< 0$, by Lemma \ref{lem:regH} our sequence 
is a.s. regular and converges to
a random variable $Z^*_{\infty} \in \partial^*\Hb = \R$ almost surely
in the topology of $\wh\Hb$. But then, using \eqref{eq:hyp-conf} as in
Lemma \ref{lem:reduce}, 
$$
\begin{aligned}
\dist_{\Hb}(Z_{\tau(n)},\im) = \dist_{\Hb}(Z_{\tau(n)}^{-1},\im) &\asymplaw
\dist_{\Hb}(Z^*_{\tau(n)},\im)\\
&\hspace*{7pt}\asymp \hspace*{6pt} 2 \log \IM( Z^*_{\infty} \wedge \im)
- \log \IM Z^*_{\tau(n)} \asymp Y_{\tau(n)}\,.
\end{aligned}
$$
where $\asymp$ means that the difference between the left and right hand
sides is bounded in absolute value.

Therefore, combining Lemma \ref{lem:reduce}(a) with Corollary \ref{cor:inter}, 
$$
\frac{1}{\sqrt{t}}\dist_{\HT}(X_t\,,o) \assim 
\frac{1}{\sqrt{t}}\dist_{\Hb}(Z_{\tau(\nn_t)}\,,\im) \simlaw
\frac{1}{\sqrt{t}} Y_{\tau(\nn_t)} 
\assim \frac{1}{\sqrt{t}}Y_t\,,
$$
as $t \to \infty$. Now Lemma \ref{lem:CLTY} yields the result, when
$\ell(\al,\beta) > 0$.
\\[5pt]
\emph{Case 2.} $\;\ell(\al,\beta) < 0$.
\\
Here, Lemma \ref{lem:reduce}(b) tells us that $\dist_{\HT}(X_t,\of)/\sqrt{t}$
behaves in law like $\dist_{\T}(W_t,o)/\sqrt{t}$ on $\T$, which in turn in view of 
Lemma \ref{lem:gRno}
behaves like $\dist_{\T}(R_{\nn_t}o,o)/\sqrt{t}$. One can proceed as in Case 1.
This time we can use the proof of the CLT for $(R_no)$ that is given in
\cite{CKW}; we get that 
$$
\dist_{\T}(R_{\nn_t}o,o)/\sqrt{t} \simlaw -\Phi(R_{\nn_t})/\sqrt{t} 
\sim -Y_t/\sqrt{t},$$  
and the result follows.
\end{proof}

The central limit theorem in the drift-free case requires some subtle 
input from \cite{Gr1}, \cite{Gr2}; it will be modelled after \cite{Ber}, which in
turn relies on \cite{CKW}. (In \cite{Ber}, the weak limit is incorrect, due
to a small error in the last step.)

We need standard Brownian motion $(\mathcal B_t)_{t \ge 0}$ on $\R$
starting at $0$, and the associated random variables 
$$
\ul{\mathcal M} = \min \{ \mathcal B_t : 0 \le t \le 1\}\,,\quad 
\ol{\mathcal M} = \max \{ \mathcal B_t : 0 \le t \le 1\}\,, \AND
\mathcal N = \mathcal B_1\,,
$$
so that $\mathcal N$ has standard normal distribution.

\begin{thm}\label{thm:CLT-driftfree} 
If $\ell(\al,\beta) = 0$ then
$$
\frac{1}{\sqrt t} \dist_{\HT}(X_t\,,\of) \to \frac{\log\qq}{\sqrt{\Ex(\tau)}}
\,\Bigl( 2\ol{\mathcal M} - 2\ul{\mathcal M} - |\mathcal N | \Bigr)\quad 
\text{in law, as\;}t \to \infty\,.
$$
\end{thm}

\begin{proof} In the proof, we suppose that $(X_t)$ starts at $\of$, so that
$Z_0=\im$, $W_0=o$ and $Y_0 =0$. The passage to arbitrary starting point is
a simple exercise that we leave to the reader. 
By Corollary \ref{cor:inter} and \eqref{eq:nnt:t}, 
$$
\frac{1}{\sqrt t} \dist_{\HT}(X_t\,,\of) 
\sim \frac{1}{\sqrt t} \dist_{\HT}(X_{\tau(n_t)}\,,\of)
\sim \frac{1}{\Ex(\tau)}\frac{1}{\sqrt {n_t}} \dist_{\HT}(X_{\tau(n_t)}\,,\of)
\quad\text{almost surely,} 
$$
as $t \to \infty\,$. Thus, we want to show that 
$\dist_{\HT}(X_{\tau(n)}\,,\of)\big/\sqrt{n} \to (\log\qq) U_0$
in law, as $n \to \infty$. 
By Proposition \ref{pro:metric}, 
\begin{equation}\label{eq:decompo}
\frac{1}{\sqrt n} \dist_{\HT}(X_{\tau(n)}\,,\of) \sim
\frac{1}{\sqrt n} \dist_{\Hb}(Z_{\tau(n)}\,,\im) +  
\frac{\log\qq}{\sqrt n} \dist_{\T}(W_{\tau(n)}\,,o) -
\frac{\log\qq}{\sqrt n}\bigl|Y_{\tau(n)}\bigr|.
\end{equation}
By Lemma \ref{lem:gRno}, we can identify
$$
X_{\tau(n)} = \RR_n\of\,,\quad W_{\tau(n)} = R_no\,,\AND
Y_{\tau(n)} = \Phi(R_n).
$$
In the drift-free case, $\bigl(\Phi(R_n)\bigr)$ 
is nothing but classical simple random walk on $\Z$ starting at $0$. 
Define 
$$
\begin{gathered}
\ol M_n = \max \{ \Phi(R_k) : k = 0,\dots, n\} \AND
\ul M_n = \min \{ \Phi(R_k) : k = 0,\dots, n\}\,,\\ 
\ol T(n) = \max \{ k \le n : \Phi(R_k) = \ol M_n\} \AND 
\ul T(n) = \max \{ k \le n : \Phi(R_k) = \ul M_n\}
\end{gathered}
$$
It is well known that  
\begin{equation}\label{eq:inlaw}
\frac{1}{\sqrt n}\bigl(\,\ol M_n\,, \ul M_n\,, \Phi(R_n)\bigr) 
\to \bigl(\,\ol{\mathcal M}\,,\ul{\mathcal M}\,,\mathcal N\bigr) \quad\text{in law.}
\end{equation}
See e.g. {\sc Billingsley}~\cite[(9.2)+(9.8)]{Bi} for this result and
the joint distribution of the limiting triple, and {\sc Borodin 
and Salminen}~\cite[1.15.8(2) on page 174]{BoSa} for the joint distribution of 
$\bigl(\,\ol{\mathcal M} - \ul{\mathcal M}\,, \mathcal N  \bigr)$.

\smallskip

Each $Z_{\tau(n)}$ is an element of $\Aff(\Hb,\qq)$ and can be
inverted in that group. 
Since we are assuming that 
$Z_0 = \im$, all increments $Z_{\tau(n-1)}^{-1}Z_{\tau(n)}\,$, $n \ge 1$,
are i.i.d. and have the distribution $\wt\mu$ of \eqref{eq:wtmu}. The support
of $\wt\mu$ generates the whole of $\Aff(\Hb,\qq)$, and it has finite moments 
of exponential
order by Corollary \ref{cor:mumom}. We can now invoke the method and result of
\cite{Gr1}. Since its reformulation in our setting is not completely
transparent, we provide a brief ``translation''. \cite{Gr1} comprises the
following, where we set $\ol\ttau(n) = \tau\bigl(\ol T(n)\bigr)$.

\begin{itemize}
\item
The sequence of pairs of random variables  
$$
\bigl(Z_{\ol\ttau(n)}^{-1}\,,Z_{\ol\ttau(n)}^{-1}Z_{\tau(n)}\bigr)
$$
with values in $\Aff(\Hb,\qq) \times \Aff(\Hb,\qq)\equiv \LH \times \LH$
converges in law (i.e., weakly) in $\wh \Hb \times \wh\Hb$
to a pair of independent random variables $(Z^{\dag}, Z^{\ddag})$ with values in 
$\partial^*\Hb \times \partial^*\Hb \equiv \R^2$, both of which have continuous
distributions (in fact, continuous densities with respect to Lebesgue measure).
Thus, $Z^{\dag} \ne Z^{\ddag}$ with probability $1$, so that 
$$
Z_{\ol\ttau(n)}^{-1}\wedge Z_{\ol\ttau(n)}^{-1}Z_{\tau(n)}
\tolaw Z^{\dag} \wedge Z^{\ddag} \in \Hb\,.
$$ 
\end{itemize}

Note that
$$
\log \IM \bigl(Z_{\ol \ttau(n)}^{-1}\bigr) = -(\log \qq) \ol M_n \AND 
\log \IM \bigl(Z_{\ol \ttau(n)}^{-1}Z_{\tau(n)}\bigr) = 
(\log \qq) \bigl( \Phi(R_n) - \ol M_n\bigr).
$$
Using \eqref{eq:hyp-conf}, we get that
$$
\begin{aligned}
\dist_{\Hb}(Z_{\tau(n)}\,,\im) &=
\dist_{\Hb}\bigl(Z_{\ol\ttau(n)}^{-1}\,,\, Z_{\ol\ttau(n)}^{-1}Z_{\tau(n)}\bigr)\\ 
&\asymp 2 \log \IM(Z_{\ol\ttau(n)}^{-1} \wedge Z_{\ol\ttau(n)}^{-1}Z_{\tau(n)}) 
- \log \IM \bigl(Z_{\ol\ttau(n)}^{-1}\bigr)
-\log \IM \bigl(Z_{\ol\ttau(n)}^{-1}Z_{\tau(n)}\bigr)\\
&\hspace{-6pt}\simlaw \IM(Z^{\dag} \wedge Z^{\ddag}) + 
(\log \qq)\bigl( 2 \ol M_n - \Phi(R_n) \bigr)\,.
\end{aligned}
$$
On the tree, similarly to the above, the following is proved in \cite{CKW}.
\begin{itemize}
\item
The sequence of pairs of random variables  
$$
\bigl(R_{\ul T(n)}^{-1}o\,,R_{\ul T(n)}^{-1}R_n o\bigr)
$$
with values in $\T \times \T$
converges in law (i.e., weakly) in $\wh \T \times \wh\T$
to a pair of independent random variables $(R^{\dag}, R^{\ddag})$ with values in 
$\partial^*\T \times \partial^*\T $, both of which have continuous
distributions ($\equiv$ without point masses).
Thus, $R^{\dag} \ne R^{\ddag}$ with probability $1$, so that 
$$
R_{\ul T(n)}^{-1}o \cf R_{\ul T(n)}^{-1}R_n o
\tolaw R^{\dag} \cf R^{\ddag} \in \T\,.
$$ 
\end{itemize}
This time we use \eqref{eq:tree-conf}. Noting that
$$
\hor(R_{\ul T(n)}^{-1}o) = -\ul M_n \AND \hor(R_{\ul T(n)}^{-1}R_no) =
\Phi(R_n) - \ul M_n\,,
$$
we get
$$
\begin{aligned}
\dist_{\T}(R_no\,,o) &=
\dist_{\T}\bigl(R_{\ul T(n)}^{-1}o\,,\, R_{\ul T(n)}^{-1}R_no\bigr)\\ 
&= \hor(R_{\ul T(n)}^{-1}o) + \hor(R_{\ul T(n)}^{-1}R_no) 
- 2\hor(R_{\ul T(n)}^{-1}o\cf R_{\ul T(n)}^{-1}R_no)\\
&\hspace{-6pt}\simlaw  \Phi(R_n) - 2 \ul M_n-2 \hor(R^{\dag} \wedge R^{\ddag})\,. 
\end{aligned}
$$
Putting things together (which is legitimate because our discrete time processes
are all modelled via $\RR_n$ on the same probability space), we get from
\eqref{eq:decompo} 
$$
\frac{1}{\sqrt n} \dist_{\HT}(X_{\tau(n)}\,,\of) \simlaw
\frac{\log \qq}{\sqrt n} \bigl( 2 \ol M_n - 2 \ul M_n  - |\Phi(R_n)|\bigr)\,.  
$$
Now \eqref{eq:inlaw} yields the theorem.
\end{proof}

{\bf Acknowledgement.} We acknowledge helpful interaction with Sara Brofferio,
as well as valuable hints and corrections by the referee.



\projects{\noindent Partially supported by Austrian Science Fund projects FWF-P19115, 
 FWF-P24028 and W1230, by NAWI Graz, and by the ESF programme RGLIS.

  A. Bendikov was supported 
by the Polish Government Scientific Research Fund, Grant 2012/05/B/ST 1/00613 and by SFB 701 of 
the German Research Council DFG.}

\begin{thebibliography}{99}
\parskip=0.7pt






\bibitem{BaWo} \RMIauthor{Bartholdi, L. and Woess, W.} \RMIpaper{Spectral computations on 
lamplighter groups and Diestel-Leader graphs} \RMIjournal{J. Fourier Anal. Appl.} 
{\bf 11}, (2005) 175--202.

\bibitem{BNW} \RMIauthor{Bartholdi, L., Neuhauser, M. and Woess, W.} \RMIpaper{Horocyclic 
products of trees} \RMIjournal{J. European Math. Soc.} {\bf 10}, (2008) 771--816. 

\bibitem{BeSa} \RMIauthor{Bendikov, A. and Saloff-Coste, L.}
\RMIpaper{Smoothness and heat kernels on metric graphs.} In preparation.

\bibitem{BSSW} \RMIauthor{Bendikov, A., Saloff-Coste, L., Salvatori, M. and Woess, W.} 
\RMIpaper{The heat semigroup and Brownian motion on strip complexes} 
\RMIjournal{Advances in Math.} {\bf 226} (2011), 992--1055.

\bibitem{BSSW3} \RMIauthor{Bendikov, A., Saloff-Coste, L., Salvatori, M. and Woess, W.} 
\RMIpaper{Brownian motion on treebolic space: positive harmonic functions} 
in preparation.

\bibitem{Ber} \RMIauthor{Bertacchi, D.} \RMIpaper{Random walks on Diestel-Leader 
graphs} \RMIjournal{Abh. Math. Sem. Univ. Hamburg} {\bf 71} (2001), 205--224.

\bibitem{Bi} \RMIauthor{Billingsley, P.}   \RMIbook{Convergence of probability measures} 2nd
ed., Wiley, New York 1999.

\bibitem{BoSa} \RMIauthor{Borodin, A. N. and Salminen, P.}   \RMIbook{Handbook of Brownian motion
-- facts and formulae} 2nd ed., Birkh\"auser, Basel, 2002.

\bibitem{BrKi} \RMIauthor{Brin, M. and Kifer, Y.} \RMIpaper{Brownian motion, harmonic 
functions and hyperbolicity for Euclidean complexes} \RMIjournal{Math. Z.} {\bf 237} 
(2001), 421--468.

\bibitem{Br} \RMIauthor{Brofferio, S.} \RMIpaper{How a centred random walk on the affine
group goes to infinity} \RMIjournal{Ann. Inst. H. Poincar\'e Probab. Statist.} {\bf 39} 
(2003) 371--384.

\bibitem{BSW} \RMIauthor{Brofferio, S., Salvatori, M. and Woess, W.} 
\RMIpaper{Brownian motion and harmonic functions on $\Sol(\pp,\qq)$} \RMIjournal{Int. Math. 
Res. Notes (IMRN)} {\bf 22} (2012), 5182--5218.

\bibitem{BrWo1} \RMIauthor{Brofferio, S. and Woess, W.} 
\RMIpaper{Green kernel estimates and the full Martin boundary for random walks
on lamplighter groups and Diestel-Leader graphs} \RMIjournal{Annales Inst. H. Poincar\'e 
Probab. Statist.} {\bf 41} (2005), 1101--1123.  

\bibitem{BrWo2} \RMIauthor{Brofferio, S. and Woess, W.} 
\RMIpaper{Positive harmonic functions for semi-isotropic random walks on trees, 
lamplighter groups, and DL-graphs} \RMIjournal{Potential Analysis} {\bf 24} (2006), 245--265.

\bibitem{Bu} \RMIauthor{Buraczewski, D.} \RMIpaper{On invariant measures of stochastic recursions in a 
critical case} \RMIjournal{Ann. Appl. Probab.} {\bf 17} (2007), 1245–-1272. 

\bibitem{BBD}  \RMIauthor{Buraczewski, D., Brofferio., S. and Damek, E.} 
\RMIpaper{On the invariant measure of the random difference equation 
$X_n=A_n X_{n-1}+ B_n$ in the critical case} \RMIjournal{Ann. Inst. H. Poincar\'e Probab.
Statist.} {\bf 48} (2012), 377--395.

\bibitem{Cat} \RMIauthor{Cattaneo, Carla} \RMIpaper{The spectrum of the continuous
Laplacian on a graph} \RMIjournal{Monatsh. Math.}  {\bf124} (1997), 215--235.

\bibitem{CKW} \RMIauthor{Cartwright, D. I., Kaimanovich, V. A. and Woess, W.} 
\RMIpaper{Random walks on the affine group of local fields and of homogeneous 
trees} \RMIjournal{Ann. Inst Fourier (Grenoble)} {\bf 44} (1994), 1243--1288.

\bibitem{CDP} \RMIauthor{Coornaert, M., Delzant, T. and Papadopoulos, A.} 
  \RMIbook{G\'eom\'etrie et th\'eorie des groupes: les groupes hyperboliques de 
Gromov} Lecture Notes in Math. {\bf 1441}, Springer, Berlin, 1990. 

\bibitem{DiLe} \RMIauthor{Diestel, R. and Leader, I.} \RMIpaper{A conjecture 
concerning a limit of non-Cayley graphs} \RMIjournal{J. Algebraic Combin.}  {\bf 14}  
(2001), 17--25. 
 
\bibitem{EeFu} \RMIauthor{Eells, J. and Fuglede, B.} 
  \RMIbook{Harmonic maps between Riemannian polyhedra} 
Cambridge Tracts in Mathematics {\bf 142}, Cambdrige University Press, 
Cambridge, 2001.

\bibitem{EFW1} \RMIauthor{Eskin, A., Fisher, D. and Whyte, K.} \RMIpaper{Quasi-isometries
and rigidity of solvable groups} \RMIjournal{Pure Appl. Math. Q.}  {\bf 3} (2007), 927--947.

\bibitem{EFW2} \RMIauthor{Eskin, A., Fisher, D. and Whyte, K.} \RMIpaper{Coarse
differentiation of quasi-isometries I: spaces non quasi-isometric to
Cayley graphs} \RMIjournal{Annals of Math.} {\bf 176} (2012), 221--260.

\bibitem{EFW3} \RMIauthor{Eskin, A., Fisher, D. and Whyte, K.} \RMIpaper{Coarse 
differentiation of quasi-isometries II: rigidity for Sol and Lamplighter 
groups} \RMIjournal{Annals of Math.} {\bf 177} (2013), 869--910.

\bibitem{FaMo} \RMIauthor{Farb, B. and Mosher, L.} \RMIpaper{A rigidity theorem for the solvable 
Baumslag-Solitar groups. With an appendix by Daryl Cooper} \RMIjournal{Invent. Math.} {\bf 131}
(1998), 419--451.

\bibitem{GH} \RMIauthor{Ghys, E. and de la Harpe, P. {\rm (eds.)}}   \RMIbook{Sur les groupes 
hyperboliques d'apr\`es Mikhael Gromov} 
Progress in Math. {\bf 83}, Birkh\"auser, Basel, 1990.

\bibitem{Gr1} \RMIauthor{Grincevi\v cjus, A. K.} \RMIpaper{A central limit theorem
for the group of linear transformations of the line} \RMIjournal{Dokl. Akad. Nauk SSSR}
{\bf 219} (1974), 23--26 (Russian); English translation: 
\RMIjournal{Soviet Math. Doklady} {\bf 15} (1974), 1512--1515. 

\bibitem{Gr2} \RMIauthor{Grincevi\v cjus, A. K.} \RMIpaper{Limit theorems for products
of random linear transformations of a straight line} \RMIjournal{Lithuanian Math. J.} 
{\bf 15} (1975), 61--77 (Russian). 

\bibitem{Gro} \RMIauthor{Gromov, M.} \RMIpaper{Hyperbolic groups} In   \RMIbook{Essays in
group theory} (ed. S. M. Gersten), Math. Sci. Res. Inst. Publ. {\bf 8},
pp. 75--263, Springer, New York, 1987. 

\bibitem{GKR} \RMIauthor{Guivarc'h, Y., Keane, M. and Roynette, B.}   \RMIbook{Marches 
al\'eatoires sur les groupes de Lie} Lecture Notes in Mathemathics {\bf 624}, 
Springer, Berlin, 1977.

\bibitem{Kai} \RMIauthor{Kaimanovich, V. A.} \RMIpaper{Lyapunov exponents, symmetric spaces 
and a multiplicative ergodic theorem for semisimple Lie groups} \RMIjournal{Zap. 
Nauchn. Sem. Leningrad. Otdel. Mat. Inst. Steklov. (LOMI)} {\bf 164} 
(1987), \RMIbook{Differentsialnaya Geom. Gruppy Li i Mekh. IX} 29--46, 196--197 (Russian);
English translation: \RMIjournal{J. Soviet Math.} {\bf 47} (1989), 2387--2398.

\bibitem{KaLe1} \RMIauthor{Karlsson, A. and Ledrappier, F.} \RMIpaper{Propri\'et\'e de 
Liouville et vitesse de fuite du mouvement brownien} \RMIjournal{C. R. Math. Acad. Sci. Paris} 
{\bf 344} (2007), 685--690.

\bibitem{KaLe2} \RMIauthor{Karlsson, A. and Ledrappier, F.} \RMIpaper{Linear drift and Poisson 
boundary for random walks} \RMIjournal{Pure Appl. Math. Q.} {\bf 3} (2007), 1027--1036.
 
\bibitem{KeLe} \RMIauthor{Keller, M. and Lenz, D.} \RMIpaper{Dirichlet forms and stochastic completeness 
of graphs and subgraphs} \RMIjournal{J. Reine Angew. Math.} {\bf 666} (2012), 189--223.

\bibitem{SoWo} \RMIauthor{Soardi, P. M. and Woess, W.} \RMIpaper{Amenability, 
unimodularity, and the spectral radius of random walks on infinite graphs} 
\RMIjournal{Math. Z.} {\bf 205}  (1990), 471--486. 

\bibitem{W-israel} \RMIauthor{Woess,  W.} \RMIpaper{Boundaries of random walks on graphs 
and groups with infinitely many ends} \RMIjournal{Israel J. Math.} {\bf 68} (1989) 271--301. 

\bibitem{Wrec} \RMIauthor{Woess,  W.} \RMIpaper{Topological groups and recurrence
of quasi transitive graphs} \RMIjournal{Rendiconti Sem. Mat. Fis. Milano} (1994), 
185--213.

\bibitem{Wlamp} \RMIauthor{Woess,  W.} \RMIpaper{Lamplighters, Diestel-Leader graphs, 
random walks, and harmonic functions} \RMIjournal{Combinatorics, Probability \& Computing}
{\bf 14} (2005), 415--433.

\bibitem{Wo-markov} \RMIauthor{Woess,  W.}   \RMIbook{Denumerable Markov chains. 
Generating functions, boundary theory, random walks on trees} 
EMS Textbooks in Mathematics. European Mathematical Society, 
Z\"urich, 2009.


\end{thebibliography}
\end{document}